\documentclass[10pt, oneside, a4paper]{scrartcl}

\usepackage[noadjust]{cite}

\usepackage{amsmath}
\usepackage{amssymb}
\usepackage{amsthm}

\usepackage{dsfont}

\usepackage{tikz}
  \usetikzlibrary{positioning}

\usepackage{pgfplots}
  \pgfplotsset{compat = 1.13,
    colormap name = viridis,
    unbounded coords = jump,
    every axis plot/.append style = {%
      line width = 1.5pt
    }
  }
  \usepgfplotslibrary{external}
  \tikzset{external/system call = {%
    pdflatex \tikzexternalcheckshellescape
      -halt-on-error
      -interaction=batchmode
      -jobname "\image" "\texsource"}}
  \tikzexternaldisable

\usepackage{caption}
\usepackage[labelfont = bf]{subcaption}
  
\definecolor{myRed}{HTML}{E34A33}
\definecolor{myBlue}{HTML}{0571B0}
\definecolor{myBrown}{HTML}{A6611A}

\definecolor{matlabBlue}{HTML}{0072BD}
\definecolor{matlabRed}{HTML}{D95319}
\definecolor{matlabYellow}{HTML}{EDB120}
\definecolor{matlabPurple}{HTML}{7E2F8E}
\definecolor{matlabGreen}{HTML}{77AC30}
\definecolor{matlabCyan}{HTML}{4DBEEE}
\definecolor{matlabBrown}{HTML}{A2142F}

\definecolor{linkBlue}{HTML}{0055C9}
\definecolor{linkRed}{HTML}{FF1A24}
\definecolor{linkPurple}{HTML}{6200D9}

\pgfplotscreateplotcyclelist{plotlist}{
  {black, solid},
  {matlabBlue, dashed},
  {matlabBrown, dash dot},
  {matlabGreen, dotted}
}

\usepackage{filemod}

\usepackage[%
  colorlinks = true,
  linkcolor  = linkBlue,
  citecolor  = linkRed,
  urlcolor   = linkPurple
]{hyperref}
\usepackage{url}
\usepackage[nameinlink]{cleveref}

\crefformat{equation}{\textup{#2(#1)#3}}
\crefrangeformat{equation}{\textup{#3(#1)#4--#5(#2)#6}}
\crefmultiformat{equation}{\textup{#2(#1)#3}}{ and \textup{#2(#1)#3}}
{, \textup{#2(#1)#3}}{, and \textup{#2(#1)#3}}
\crefrangemultiformat{equation}{\textup{#3(#1)#4--#5(#2)#6}}%
{ and \textup{#3(#1)#4--#5(#2)#6}}{, \textup{#3(#1)#4--#5(#2)#6}}{, and \textup{#3(#1)#4--#5(#2)#6}}

\usepackage{geometry}
\usepackage{enumitem}

\usepackage{todonotes}

\newcommand{\cB}{\ensuremath{\mathcal{B}}}
\newcommand{\cC}{\ensuremath{\mathcal{C}}}

\newcommand{\cK}{\ensuremath{\mathcal{K}}}
\newcommand{\cN}{\ensuremath{\mathcal{N}}}
\newcommand{\hcB}{\ensuremath{\widehat{\mathcal{B}}}}
\newcommand{\hcC}{\ensuremath{\widehat{\mathcal{C}}}}
\newcommand{\hcK}{\ensuremath{\widehat{\mathcal{K}}}}
\newcommand{\hcN}{\ensuremath{\widehat{\mathcal{N}}}}
\newcommand{\hG}{\ensuremath{\widehat{G}}}
\newcommand{\hy}{\ensuremath{\hat{y}}}
\newcommand{\tcN}{\ensuremath{\widetilde{\mathcal{N}}}}

\newcommand{\C}{\ensuremath{\mathbb{C}}}
\newcommand{\R}{\ensuremath{\mathbb{R}}}

\newcommand{\iu}{\ensuremath{\mathrm{i}}}

\newcommand{\trans}{\ensuremath{\mkern-1.5mu\mathsf{T}}}
\newcommand{\herm}{\ensuremath{\mathsf{H}}}

\DeclareMathOperator{\mspan}{span}

\DeclareMathOperator{\lin}{lin}

\renewcommand{\rm}[1]{\ensuremath{\mathrm{#1}}}

\theoremstyle{plain}\newtheorem{theorem}{Theorem}
\theoremstyle{plain}\newtheorem{corollary}{Corollary}

\begin{document}

\title{Structure-Preserving Interpolation of Bilinear Control %
  Systems}

\author{%
  Peter~Benner\thanks{
    Max Planck Institute for Dynamics of Complex Technical Systems,
    Sandtorstr. 1, 39106 Magdeburg, Germany.\newline
    E-mail: \texttt{\href{mailto:benner@mpi-magdeburg.mpg.de}%
      {benner@mpi-magdeburg.mpg.de}}
    \newline
    Otto von Guericke University, Faculty of Mathematics,
    Universit{\"a}tsplatz 2, 39106 Magdeburg, Germany.\newline
    E-mail: \texttt{\href{mailto:peter.benner@ovgu.de}%
      {peter.benner@ovgu.de}}} \and
  Serkan~Gugercin\thanks{
    Department of Mathematics and Computational Modeling and Data 
    Analytics Division, Academy of Integrated Science, Virginia Tech, 
    Blacksburg, VA 24061, USA.\newline
    E-mail: \texttt{\href{mailto:gugercin@vt.edu}%
      {gugercin@vt.edu}}} \and
  Steffen~W.~R.~Werner\thanks{
    Max Planck Institute for Dynamics of Complex
    Technical Systems, Sandtorstr. 1, 39106 Magdeburg, Germany.\newline
    E-mail: \texttt{\href{mailto:werner@mpi-magdeburg.mpg.de}%
      {werner@mpi-magdeburg.mpg.de}}}
}

\date{~}

\maketitle


\begin{abstract}
  In this paper, we extend the structure-preserving interpolatory model 
  reduction framework, originally developed for linear systems, to structured  
  bilinear control systems.
  Specifically, we give explicit construction formulae for the model reduction 
  bases to satisfy different types of interpolation conditions.
  First, we  establish the analysis  for transfer function interpolation for
  single-input single-output structured bilinear systems.
  Then, we extend these results to the case of multi-input multi-output 
  structured bilinear systems by matrix interpolation.
  The effectiveness of our structure-preserving approach is illustrated by means 
  of various numerical examples.
 
  \vspace{1em}
  \noindent\textbf{Keywords:}
    model reduction, bilinear systems, structure-pre\-ser\-ving 
    approximation, structured interpolation.
    
  \vspace{1em}
  \noindent\textbf{AMS subject classifications:}
    30E05, 34K17, 65D05, 93C10, 93A15, 93C35.
\end{abstract}


\section{Introduction}%
\label{sec:intro}

The modeling of various real-world applications and processes results in 
dynamical control systems usually including nonlinearities.
Since linear approximations are very often  incapable of capturing all the 
features of nonlinear systems, they are an insufficient description for use in
optimization and controller design.
A special class of nonlinear systems are bilinear control systems, which contain
the multiplication of control and state variables, i.e., they are linear in
state and control separately, but not together~\cite{Moh70}.
In the last decades, the class of bilinear systems became an essential tool
in systems theory.
They naturally appear in the modeling process of many physical phenomena, e.g.,
in the modeling of population, economical, thermal and mechanical 
dynamics~\cite{Moh70, Moh73}, of electrical circuits~\cite{morAlbFB93}, of plasma
devices~\cite{Ou10, QiaZ14}, or of medical 
processes~\cite{SapSH19}.
Bilinear systems can also result  from approximation of general nonlinear 
systems employing the Carleman linearization process~\cite{Car32, KowS91}.
Moreover, bilinear systems are nowadays often used in the parameter control
of partial differential equations (PDEs)~\cite{Kha03, KorK18}.
Looking back to the linear case, bilinear systems can be used as a generalizing 
framework in the modeling of linear stochastic~\cite{morBenD11} and 
parameter-varying systems~\cite{morBenB11, morBruB15, morBenCS19}, allowing the 
application of established system-theoretic tools such as model order reduction 
for those system classes.

In this paper, we focus on structured bilinear systems.
Those structures arise from the underlying physical phenomena. For example,  in 
case of bilinear mechanical systems, one has the bilinear control system 
defined by
\begin{align} \label{eqn:bsosys}
  \begin{aligned}
    M\ddot{q}(t) + D\dot{q}(t) + Kq(t) & = \sum\limits_{j = 1}^{m}
      N_{\rm{p},j} q(t) u_{j}(t) + \sum\limits_{j = 1}^{m} 
      N_{\rm{v},j} \dot{q}(t) u_{j}(t) + B_{\rm{u}} u(t),\\
    y(t) & = C_{\rm{p}}q(t) + C_{\rm{v}}\dot{q}(t),
  \end{aligned}
\end{align}
with $M, D, K, N_{\rm{p},j}, N_{\rm{v},j} \in \R^{n \times n}$ for all
$j = 1, \ldots, m$, $B_{\rm{u}} \in \R^{n \times m}$ and $C_{\rm{p}}, C_{\rm{v}}
\in \R^{p \times n}$. 
In~\cref{eqn:bsosys}, $q(t) \in \R^{n}$, $u(t) \in \R^m$, and $y(t) \in \R^p$ 
denote, respectively, the states (degrees of freedom), inputs (forcing terms), 
and the outputs (quantities of interest) of the underlying dynamical system.
Due to the usual demand for increasing accuracy in applications, the number of
differential equations $n$, describing the dynamics of systems
as in~\cref{eqn:bsosys} quickly increases,  resulting in a high demand on
computational resources such as time and memory.
One remedy is model order reduction: a new, \emph{reduced}, system is created, 
consisting of a significantly smaller number of differential equations than 
needed to define the original one while still accurately approximating the 
input-to-output behavior.
Then one can use this lower-order approximation as a surrogate model for faster 
simulations or the design of controllers.
The classical (unstructured) bilinear first-order systems are described by the 
state-space representation
\begin{align} \label{eqn:bsys}
  \begin{aligned}
    E \dot{x}(t) & = A x(t) + \sum\limits_{j = 1}^{m}N_{j}x(t)u_{j}(t)
     + B u(t),\\
    y(t) & = C x(t),
  \end{aligned}
\end{align}
with $E, A, N_{j} \in \R^{n \times n}$ for all $j = 1, \ldots, m$,
$B \in \R^{n \times m}$ and $C \in \R^{p \times n}$.
There are different methodologies for model reduction of~\cref{eqn:bsys}, e.g.,
the bilinear balanced truncation method~\cite{morHsuDC83, morAlbFB93,
morBenD11}, different types of moment matching approaches for the underlying
multi-variate transfer functions in the frequency domain~\cite{morBaiS06,
morConI07, morFenB07a, morBreD10, morAntBG20}, the interpolation of
complete Volterra series~\cite{morZhaL02, morBenB12b, morFlaG15} or even the
construction of reduced-order bilinear systems from frequency data with the 
bilinear Loewner framework~\cite{morAntGI16, morGosPBetal19}.

While it is possible to rewrite~\cref{eqn:bsosys} into a
classical bilinear system~\cref{eqn:bsys}, the original structure is
completely lost, which can lead to undesirable results in terms of accuracy,
stability or physical interpretation. Moreover, some other structured bilinear 
systems, such as those with internal delays (see \Cref{sec:time-delay-models}), 
cannot be represented in the form~\cref{eqn:bsys}.
Therefore, here we develop a structure-preserving model reduction
approach for different system structures involving bilinear terms.
Following~\cite{morBeaG09}, which studied structured linear dynamical 
systems, our goal is to generalize the structured
interpolation approach to a general set of multivariate transfer functions
associated with different structured bilinear control systems to preserve the 
system structure in the reduced-order model.
The question we aim to answer is how  we can construct an interpolatory
reduced-order model of, e.g.,~\cref{eqn:bsosys}, that has the same structure.
Towards this goal, we develop a structure-preserving interpolation framework for 
this special class of nonlinear systems, namely the structured bilinear 
control systems; thus extending the theoretical analysis and computational 
framework developed by~\cite{morBeaG09} for linear systems to bilinear control 
systems.

In \Cref{sec:basics}, we review the theory for classical first-order bilinear 
systems and motivate the more general structure, we will consider, via two 
examples.
\Cref{sec:siso} gives subspace construction formulae for interpolatory 
model reduction bases in the case of single-input single-output (SISO) systems
and illustrates the effectiveness of the approach employing two numerical 
examples.
The developed theory is then  extended further in \Cref{sec:mimo}
to the multi-input multi-output (MIMO) case by matrix interpolation.
\Cref{sec:conclusions} concludes the paper.


\section{Structured bilinear systems}%
\label{sec:basics}

In this section, we present the basic properties of the structured bilinear
systems considered in this paper. To clarify the presentation, we first
revisit the unstructured (classical) bilinear  control systems as given
in~\cref{eqn:bsys} and then generalize these concepts to the structured case.


\subsection{Revisiting the classical first-order bilinear systems}%
\label{sec:classical}

Given the unstructured bilinear system~\cref{eqn:bsys}, define $N = 
\begin{bmatrix} N_{1} & \ldots & N_{m} \end{bmatrix}$ and let $I_{m^{k}}$ be 
the identity matrix of dimension $m^{k}$.
Assuming for simplicity $E$ to be invertible, the initial condition $x(0) = 0$, 
and some additional mild conditions, the output of~\cref{eqn:bsys} can be 
expressed in terms of a Volterra series~\cite{Rug81}, i.e.,
\begin{align*}
  y(t) & = \sum\limits_{k = 1}^{\infty} \int\limits_{0}^{t}
    \int\limits_{0}^{t_{1}} \ldots \int\limits_{0}^{t_{k-1}}
    g_{k}(t_{1}, \ldots, t_{k}) 
    \left( u(t - \sum\limits_{i = 1}^{j}t_{i}) \otimes \cdots \otimes
    u(t - t_{1}) \right) \mathrm{d}t_{k} \cdots \mathrm{d}t_{1},
\end{align*}
where $g_{k}$, for $k \geq 1$, is the $k$-th regular Volterra kernel given by
\begin{align} \label{eqn:voltime}
  \begin{aligned}
    g_{k}(t_{1}, \ldots, t_{k}) & = Ce^{E^{-1}At_{k}} \left(
      \prod\limits_{j = 1}^{k-1} (I_{m^{j-1}} \otimes E^{-1}N)
      (I_{m^{j}} \otimes e^{E^{-1}At_{k-j}})
      \right)\\
    & \quad{}\times{} (I_{m^{k-1}} \otimes E^{-1}B).
  \end{aligned}
\end{align}
Using the multivariate Laplace transform~\cite{Rug81}, the regular Volterra
kernels~\cref{eqn:voltime} yield a representation of~\cref{eqn:bsys}
in the frequency domain by the so-called multivariate regular transfer functions
\begin{align} \label{eqn:btf}
  \begin{aligned}
    G_{k}(s_{1},\ldots,s_{k}) & = C(s_{k}E - A)^{-1} \left(
      \prod\limits_{j = 1}^{k-1} (I_{m^{j-1}} \otimes N)
      (I_{m^{j}} \otimes (s_{k-j}E - A)^{-1})
      \right)\\
    & \quad{}\times{} (I_{m^{k-1}} \otimes B),
  \end{aligned}
\end{align}
with $s_{1}, \ldots, s_{k} \in \C$.
This compact expression is actually the collection of the different
combinations of the bilinear matrices, i.e., we can write~\cref{eqn:btf} as
\begin{align} \label{eqn:nokron}
  \begin{aligned}
    G_{k}(s_{1},\ldots,s_{k}) & = [C(s_{k}E - A)^{-1}N_{1} \cdots 
      N_{1}(s_{1}E - A)^{-1}B,\\
    & ~~~~~ C(s_{k}E - A)^{-1}N_{1} \cdots N_{2}(s_{1}E - A)^{-1}B,\\
    & ~~~~~ \ldots \\
    & ~~~~~ C(s_{k}E - A)^{-1}N_{1} \cdots N_{m}(s_{1}E - A)^{-1}B,\\
    & ~~~~~ \ldots \\
    & ~~~~~ C(s_{k}E - A)^{-1}N_{m} \cdots N_{m}(s_{1}E - A)^{-1}B].
  \end{aligned}
\end{align}
For SISO systems,~\cref{eqn:btf} simplifies to
\begin{align*}
  G_{k}(s_{1},\ldots,s_{k}) & = C(s_{k}E - A)^{-1} \left( 
    \prod\limits_{j = 1}^{k-1} N (s_{k-j}E - A)^{-1} \right) B.
\end{align*}
As stated in \Cref{sec:intro}, for the unstructured bilinear system
case~\cref{eqn:bsys}, there are already different model reduction techniques.
For the structured bilinear systems we consider here, we will concentrate on 
interpolatory methods.

Note that the assumption of $E$ being invertible is only made for ease of 
presentation.
The interpolation theory and interpolatory properties of the reduced-order model 
developed in the following sections hold for the general situation, yet the 
final construction of the reduced-order model might need some additional 
treatment as in the linear and unstructured bilinear cases; see,
e.g.,~\cite{morGugSW13, morBenG16, morAhmBG17}.


\subsection{Moving from classical to structured bilinear systems}%
\label{sec:structuredproblems}

For the transition from unstructured to structured bilinear systems, we start by 
recalling  the case of linear systems.
The classical (unstructured) linear dynamical systems are described, in 
state-space, by
\begin{align*}
  E \dot{x}(t) & = A x(t) + B u(t),\\
  y(t) & = C x(t),
\end{align*}
with $E, A \in \R^{n \times n}$, $B \in \R^{n \times m}$ and
$C \in \R^{p \times n}$.
Assuming the initial condition $Ex(0) = 0$, the Laplace transform maps this 
problem to the frequency domain:
\begin{align} \label{eqn:freqsys}
  \begin{aligned}
    (sE - A) X(s) & = B U(s), \\
    Y(s) & = C X(s),
  \end{aligned}
\end{align}
where $X(s), U(s)$ and $Y(s)$ denote the Laplace transforms of the 
time-dependent functions $x(t)$, $u(t)$, and $y(t)$.
Inspired by much richer structured systems than~\cref{eqn:freqsys} 
appearing in the linear case such as those describing the dynamic response of a 
viscoelastic body, \cite{morBeaG09}~introduced a more general  system of 
equations in the frequency domain, given by
\begin{align} \label{eqn:structfreqsys}
  \begin{aligned}
    \cK(s) X(s) & = \cB(s) U(s), \\
    Y(s) & = \cC(s) X(s),
  \end{aligned}
\end{align}
with matrix-valued functions $\cK\colon \C \rightarrow \C^{n \times n}$, 
$\cB\colon \C \rightarrow \C^{n \times m}$ and
$\cC\colon \C \rightarrow \C^{p \times n}$.
Note that~\cref{eqn:structfreqsys} 
contains~\cref{eqn:freqsys} as a special case. 
Assuming the problem to be regular, i.e., there exists an $s \in \C$ for which
the matrix functions are defined and $\cK(s)$ is full-rank, the 
problem~\cref{eqn:structfreqsys} leads to the general formulation of structured
transfer functions of linear systems
\begin{align} \label{eqn:lintf}
  G_{\lin}(s) & = \cC(s) \cK(s)^{-1} \cB(s),
\end{align}
describing the input-to-output behavior in the frequency domain.

Inspired by~\cref{eqn:lintf} and the structure of the examples in 
\Cref{sec:2ndorderexm,sec:time-delay-models},
we consider here a more general, structured formulation of the regular
 subsystem transfer functions corresponding to structured bilinear systems, namely 
\begin{align} \label{eqn:mimotf}
  \begin{aligned}
    G_{k}(s_{1}, \ldots, s_{k}) & = \cC(s_{k})\cK(s_{k})^{-1} \left(
      \prod\limits_{j = 1}^{k-1} \big( I_{m^{j-1}} \otimes \cN(s_{k-j}) \big)
      \big( I_{m^{j}} \otimes \cK(s_{k-j})^{-1} \big)
      \right)\\
    & \quad{}\times{} (I_{m^{k-1}} \otimes \cB(s_{1})),
  \end{aligned}
\end{align}
for $k \geq 1$ and where $\cN(s) = \begin{bmatrix} \cN_{1}(s) & \ldots &
\cN_{m}(s) \end{bmatrix}$ with the matrix functions
$\cC\colon \C \rightarrow \C^{p \times n}$,
$\cK\colon \C \rightarrow \C^{n \times n}$,
$\cB\colon \C \rightarrow \C^{n \times m}$,
$\cN_{j}\colon \C \rightarrow \C^{n \times n}$
for $j = 1, \ldots, m$.
This general formulation includes transfer functions of classical bilinear
systems~\cref{eqn:btf} since we can choose
\begin{align*}
  \begin{aligned}
    \cC(s) & = C, & \cK(s) & = sE - A, & \cN(s) & = N, & \cB(s) & = B.
  \end{aligned}
\end{align*}
\Cref{sec:2ndorderexm,sec:time-delay-models} give two examples of structured 
system classes that can be formulated in this general setting.

For the construction of \emph{structured reduced-order bilinear models}, we will 
use the projection approach, i.e., we will construct two model reduction bases 
$W, V \in \C^{n \times r}$ such that the reduced-order bilinear system 
quantities will be computed by
\begin{align} \label{eqn:proj}
  \begin{aligned}
    \hcC(s) & = \cC(s) V,
      & \hcK(s) & = W^{\herm} \cK(s) V,
      & \hcB(s) & = W^{\herm} \cB(s), && \text{and} &
      & \hcN_{j}(s) & = W^{\herm} \cN_{j}(s) V,
  \end{aligned}
\end{align}
for $j = 1, \ldots, m$.
The structured reduced-order bilinear control system $\hG$ is then given by the 
underlying reduced-order matrices from~\cref{eqn:proj} and with the 
corresponding structured multivariate subsystem transfer functions
\begin{align*}
  \hG_{k}(s_{1}, \ldots, s_{k}) & = \hcC(s_{k})\hcK(s_{k})^{-1} \left(
    \prod\limits_{j = 1}^{k-1} \big( I_{m^{j-1}} \otimes \hcN(s_{k-j}) \big)
    \big( I_{m^{j}} \otimes \hcK(s_{k-j})^{-1} \big)
    \right) (I_{m^{k-1}} \otimes \hcB(s_{1})),
\end{align*}
for $k \geq 1$.
  

\subsubsection{Bilinear second-order systems}
\label{sec:2ndorderexm}

We revisit the example of second-order bilinear 
systems~\cref{eqn:bsosys} given in \Cref{sec:intro}.
First, we note that~\cref{eqn:bsosys} can be rewritten in the first-order 
(unstructured) form~\cref{eqn:bsys} by introducing the new state vector
$x(t) = [q^{\trans} (t), \dot{q}^{\trans}]^{\trans}$ such that we obtain
\begin{align} \label{eqn:bsosyscomp}
  \begin{aligned}
    \underbrace{\begin{bmatrix} J & 0 \\ 0 & M \end{bmatrix}}_{E}
      \dot{x}(t) & = \underbrace{\begin{bmatrix} 0 & J \\ -K & -D 
      \end{bmatrix}}_{A} x(t) + \sum\limits_{j = 1}^{m}\underbrace{
      \begin{bmatrix} 0 & 0 \\ N_{\rm{p},j} & N_{\rm{v},j} \end{bmatrix}}_{N_{j}}
      x(t)u_{j}(t) + \underbrace{\begin{bmatrix} 0 \\ B_{\rm{u}} 
      \end{bmatrix}}_{B}
      u(t),\\
    y(t) & = \underbrace{\begin{bmatrix} C_{\rm{p}} & C_{\rm{v}} 
      \end{bmatrix}}_{C}x(t),
  \end{aligned}
\end{align}
for any invertible matrix $J \in \R^{n \times n}$.
For this first-order companion realization~\cref{eqn:bsosyscomp}, we know the
frequency domain representation to be given by the multivariate regular
transfer functions~\cref{eqn:btf}.
If we now plug in the structured matrices from~\cref{eqn:bsosyscomp}, we can
make use of those special block structures.
In general, we obtain
\begin{align*}
  (sE - A)^{-1} & = \begin{bmatrix} sJ & -J \\ K & sM + E \end{bmatrix}^{-1}\\
  & = \begin{bmatrix} \frac{1}{s}J^{-1} - \frac{1}{s} (s^2 M + sD + K)^{-1}
    KJ^{-1} & (s^2 M + sD + K)^{-1} \\ -(s^2 M + sD + K)^{-1}KJ^{-1} & 
    s(s^2 M + sD + K)^{-1}\end{bmatrix}
\end{align*}
for the frequency-dependent center terms and, therefore,
\begin{align*}
  N_{j}(sE - A)^{-1}B & = \begin{bmatrix} 0 \\ (N_{\rm{p},j} + sN_{\rm{v},j})
    (s^{2}M + sD + K)^{-1}B_{\rm{u}} \end{bmatrix}.
\end{align*}
Using this, we obtain for the first part of the $k$-th regular transfer function
\begin{align*}
  & \left(\prod\limits_{j = 1}^{k-1} \big( I_{m^{j-1}} \otimes N \big)
    \big( I_{m^{j}} \otimes (s_{k-j}E - A)^{-1} \big)
    \right) (I_{m^{k-1}} \otimes B)\\
  & = \begin{bmatrix} 0 \\ \left(\prod\limits_{j = 1}^{k-1}
    \big( I_{m^{j-1}} \otimes (N_{\rm{p}} + s_{k-j}N_{\rm{v}}) \big)
    \big( I_{m^{j}} \otimes (s_{k-j}^{2}M + s_{k-j}D + K)^{-1} \big)
    \right) (I_{m^{k-1}} \otimes B_{\rm{u}})
    \end{bmatrix},
\end{align*}
where we used the notion $N_{\rm{p}} = \begin{bmatrix} N_{\rm{p},1} & \ldots & 
N_{\rm{p},m} \end{bmatrix}$ and $N_{\rm{v}} = \begin{bmatrix} N_{\rm{v},1} & 
\ldots & N_{\rm{v},m} \end{bmatrix}$.
Multiplication with the remaining terms yields the regular transfer
functions of~\cref{eqn:bsosys} to be written in the form
\begin{align}\label{eqn:bsotf}
  \begin{aligned}
    G_{k}(s_{1}, \ldots, s_{k}) & = (C_{\rm{p}} + s_{k}C_{\rm{v}})
      (s_{k}^{2}M + s_{k}D + K)^{-1} \left(\prod\limits_{j = 1}^{k-1}
      \big( I_{m^{j-1}} \otimes (N_{\rm{p}} + s_{k-j}N_{\rm{v}}) \big) \right. \\
    & \quad{}\times{} \left. \vphantom{\prod\limits_{j = 1}^{k-1}}
      \big( I_{m^{j}} \otimes (s_{k-j}^{2}M + s_{k-j}D + K)^{-1} \big)
      \right) (I_{m^{k-1}} \otimes B_{\rm{u}}).
  \end{aligned}
\end{align}
Having the general formulation of regular transfer functions~\cref{eqn:mimotf}
in mind, we see that we can rewrite \cref{eqn:bsotf} in the structured bilinear 
form \cref{eqn:mimotf} by setting 
\begin{align*}
  \begin{aligned}
    \cC(s) & = C_{\rm{p}} + sC_{\rm{v}}, & 
    \cK(s) & = s^{2}M + sD + K, & 
    \cN(s)  & =N_{\rm{p}} + sN_{\rm{v}}, & 
    \cB(s) & = B_{\rm{u}}.
  \end{aligned}
\end{align*}
Now assume that we construct model reduction bases $W$ and $V$ and compute
the reduced order model by projection as in~\cref{eqn:proj}.
This leads to the reduced-order bilinear system
\begin{align} \label{eqn:redsecorder}
  \begin{aligned}
    \hcC(s) & = C_{\rm{p}} V + s (C_{\rm{v}} V), \\
    \hcK(s) & = s^{2} (W^\herm M  V)+ s (W^\herm D V)+ (W^\herm K V), \\
    \hcN(s) &  = (W^\herm N_{\rm{p}} (I_{m} \otimes V))
      + s(W^\herm N_{\rm{v}} (I_{m} \otimes V)), \\
    \hcB(s) & = W^\herm B_{\rm{u}}.
  \end{aligned}
\end{align}
Note that the reduced-order bilinear system in~\cref{eqn:redsecorder} has the 
same structure as the original one and can be viewed as a reduced second-order
bilinear system, where the full-order matrices in~\cref{eqn:bsosys} are simply
replaced by the reduced analogues from~\cref{eqn:redsecorder}.


\subsubsection{Bilinear time-delay systems}
\label{sec:time-delay-models}

Another structured bilinear control system example is the case of  bilinear systems with an internal time-delay, i.e.,
\begin{align*}
  E\dot{x}(t) & = Ax(t) + A_{\rm{d}}x(t - \tau) +
    \sum\limits_{j = 1}^{m}N_{j}x(t)u_{j}(t) + Bu(t),\\
  y(t) & = Cx(t),
\end{align*}
for a delay $0 \leq \tau \in \R$, which has regular transfer
functions of the form
\begin{align} \label{eqn:tdbtf}
  \begin{aligned}
    G_{k}(s_{1}, \ldots, s_{k}) & = C(s_{k}E - A - e^{-s_{k}\tau}A_{\rm{d}})^{-1}
      \left(\prod\limits_{j = 1}^{k-1} \big( I_{m^{j-1}} \otimes N \big)
      \right. \\
    & \quad{}\times{} \left. \vphantom{\prod\limits_{j = 1}^{k-1}}
      \big( I_{m^{j}} \otimes (s_{k-j}E - A - e^{-s_{k-j}\tau}A_{\rm{d}})^{-1} \big)
      \right) (I_{m^{k-1}} \otimes B);
  \end{aligned}
\end{align}
see~\cite{morGosPBetal19}.
As in the case of the bilinear second-order systems, we see 
that~\cref{eqn:tdbtf}
can be written in the setting of~\cref{eqn:mimotf} using
\begin{align*}
  \begin{aligned}
    \cC(s) & = C, & \cK(s) & = sE - A - e^{-s\tau}A_{\rm{d}}, & \cN(s) & =
      N, && \text{and} & \cB(s) & = B.
  \end{aligned}
\end{align*}
As in \Cref{sec:2ndorderexm}, once the model reduction bases 
$W$ and $V$ are constructed, 
the resulting reduced-order
model retains the  delay structure of the original system
as it is given by
\begin{align*}
  \begin{aligned}
    \hcC(s) & = C V, & 
    \hcK(s) & = s(W^{\herm} E V) - (W^{\herm} A V) - e^{-s\tau}(W^{\herm}
      A_{\rm{d}} V), &  \\
    \hcN(s) &  = W^\herm N (I_{m} \otimes V), & \hcB(s) & = W^\herm B.
  \end{aligned}
\end{align*}
In \Cref{sec:siso,sec:mimo}, we will show how to construct the model reduction 
bases $W$ and $V$ such that the structured reduced-order bilinear control system 
provides interpolation of the full-order subsystems.


\section{Interpolation of single-input single-output systems}%
\label{sec:siso}

In this section, we assume the SISO system case, i.e., $m = p = 1$.
Therefore, the bilinear part consists of, at most, one term $\cN = \cN_{1}$ and
the matrix functionals $\cC$ and $\cB$ map frequency points only onto row
and column vectors, respectively.
In this setting, the regular transfer functions drastically simplify 
since~\cref{eqn:mimotf} can now be written as
\begin{align} \label{eqn:sisotf}
  G_{k}(s_{1}, \ldots, s_{k}) & = \cC(s_{k}) \cK(s_{k})^{-1} \left(
    \prod\limits_{j = 1}^{k-1} \cN(s_{k-j}) \cK(s_{k-j})^{-1} \right) \cB(s_{1}),
\end{align}
for $k \geq 1$.
In the remainder of this section, we develop the theory for structure-preserving 
interpolation (both the case of simple and high-order (Hermite) interpolation) 
and then present numerical examples to illustrate the analysis.


\subsection{Structured transfer function interpolation}

We want to construct the model reduction bases $W$ and $V$ and the corresponding 
reduced struc\-tured-bilinear system via projection as in~\cref{eqn:proj} 
such that its leading regular transfer functions interpolate those of the 
original one; i.e., $G_{k}(\sigma_{1}, \ldots, \sigma_{k}) =
\hG_{k}(\sigma_{1}, \ldots, \sigma_{k})$, where $\sigma_{1}, \ldots, \sigma_{k} 
\in \C$ are some selected interpolation points.

The following two theorems answer the question of how the model reduction bases
$V$ and $W$ can be constructed independent of each other.
In other words, the interpolation conditions are satisfied only via $V$ or $W$, 
no matter how the respective other matrix is chosen.  
First, we consider the model reduction basis  $V$.

\begin{theorem}[Interpolation via $V$]%
  \label{thm:sisov}
  Let $G$ be a bilinear SISO system, described by~\cref{eqn:sisotf},
  and $\hG$ the reduced-order bilinear SISO system constructed
  by~\cref{eqn:proj}. Let  $\sigma_{1},\ldots,\sigma_{k} \in 
  \C$  be interpolation points for which the matrix functions $\cC(s)$, 
  $\cK(s)^{-1}$, $\cN(s)$ and $\cB(s)$ are defined and $\hcK(s)$ is full-rank.
  Construct $V$ using
  \begin{align*}
    v_{1} & = \cK(\sigma_{1})^{-1}\cB(\sigma_{1}),\\
    v_{j} & = \cK(\sigma_{j})^{-1}\cN(\sigma_{j-1})v_{j-1}, & 2 \leq j \leq k,\\
    \mspan(V) & \supseteq \mspan\left([v_{1}, \ldots, v_{k}]\right),
  \end{align*}
  and let $W$ be an arbitrary full-rank truncation matrix of appropriate 
  dimension.
  Then the subsystem transfer functions of $\hG$ interpolate those of 
  $G$ in the following way:
  \begin{align*}
    \begin{aligned}
      G_{1}(\sigma_{1}) & =  \hG_{1}(\sigma_{1}), &
      G_{2}(\sigma_{1},\sigma_{2}) & = \hG_{2}(\sigma_{1}, \sigma_{2}), &
      & \ldots, & 
      G_{k}(\sigma_{1}, \ldots, \sigma_{k}) &
        = \hG_{k}(\sigma_{1}, \ldots, \sigma_{k}).
    \end{aligned}
  \end{align*}
\end{theorem}
\begin{proof}
  First, we note that the constructed vectors are given by
  \begin{align*}
    v_{1} & = \cK(\sigma_{1})^{-1}\cB(\sigma_{1}),\\
    v_{2} & = \cK(\sigma_{2})^{-1}\cN(\sigma_{1})\cK(\sigma_{1})^{-1}
      \cB(\sigma_{1}),\\ 
    & \,\,\,\vdots \\
    v_{k} & = \cK(\sigma_{k})^{-1}\cN(\sigma_{k-1})\cK(\sigma_{k-1})^{-1} \cdots 
      \cK(\sigma_{1})\cB(\sigma_{1}),
  \end{align*}
  and that by construction all those vectors are contained in $\mspan(V)$.
  Therefore, for the first transfer function we obtain
  \begin{align*}
    \hG_{1}(\sigma_{1}) & = \hcC(\sigma_{1}) \hcK(\sigma_{1})^{-1}
      \hcB(\sigma_{1})\\
    & = \cC(\sigma_{1})V(W^{\herm}\cK(\sigma_{1})V)^{-1}W^{\herm}
      \cB(\sigma_{1})\\
    & = \cC(\sigma_{1}) \underbrace{V (W^{\herm} \cK(\sigma_{1}) V)^{-1}
      W^{\herm} \cK(\sigma_{1})}_{\phantom{\, P_{v_{1}}} =:\, P_{v_{1}}} 
      \cK(\sigma_{1})^{-1} \cB(\sigma_{1})\\
    & = \cC(\sigma_{1}) \cK(\sigma_{1})^{-1} \cB(\sigma_{1})\\
    & = G_{1}(\sigma_{1}),
  \end{align*}
  where we used the fact that $P_{v_{1}}$ is an oblique projector onto
  $\mspan(V)$, i.e., $z = P_{v_{1}} z$ holds for all $z \in \mspan(V)$,
  and that $\cK(\sigma_{1})^{-1} \cB(\sigma_{1}) = v_{1} \in \mspan(V)$.
  Considering the second transfer function, we get
  \begin{align*}
    \hG_{2}(\sigma_{1}, \sigma_{2}) & = \hcC(\sigma_{2}) \hcK(\sigma_{2})^{-1}
      \hcN(\sigma_{1}) \hcK(\sigma_{1})^{-1} \hcB(\sigma_{1})\\
    & = \cC(\sigma_{2}) V (W^{\herm} \cK(\sigma_{2}) V)^{-1} W^{\herm}
      \cN(\sigma_{1}) V (W^{\herm} \cK(\sigma_{1}) V)^{-1} W^{\herm}
      \cB(\sigma_{1})\\
    & = \cC(\sigma_{2}) V (W^{\herm} \cK(\sigma_{2}) V)^{-1} W^{\herm}
      \cN(\sigma_{1}) \cK(\sigma_{1})^{-1} \cB(\sigma_{1})\\
    & = \cC(\sigma_{2}) \underbrace{V (W^{\herm} \cK(\sigma_{2}) V)^{-1}
      W^{\herm} \cK(\sigma_{2})}_{\phantom{\, P_{v_{2}}}=:\, P_{v_{2}}} 
      \cK(\sigma_{2})^{-1} \cN(\sigma_{1}) \cK(\sigma_{1})^{-1} \cB(\sigma_{1})\\
    & = \cC(\sigma_{2})\cK(\sigma_{2})^{-1}\cN(\sigma_{1})\cK(\sigma_{1})^{-1}
      \cB(\sigma_{1})\\
    & = G_{2}(\sigma_{1}, \sigma_{2}),
  \end{align*}
  using the same arguments as for the first transfer function and additionally
  the construction of $v_{2}$ and the oblique projector $P_{v_{2}}$.
  Continuing with this argumentation, the desired result follows by
  induction over the transfer function index $k$.
\end{proof}

The proof of \Cref{thm:sisov} shows that the recursive construction
of the truncation matrix is necessary for the interpolation of higher-order
transfer functions.
Also, it should be noted that $W$ was an arbitrary full-rank truncation matrix of
suitable dimensions but with no additional constraints for the interpolation
of~\cref{eqn:sisotf}.
\Cref{thm:sisow} is the counterpart to \Cref{thm:sisov}
by  only giving constraints for the left model reduction basis $W$, while $V$ is 
now allowed to be arbitrary.

\begin{theorem}[Interpolation via $W$]%
  \label{thm:sisow}
  Let $G$, $\hG$, and the interpolation points $\sigma_{1},\ldots,\sigma_{k} \in 
  \C$ be as in \Cref{thm:sisov}.
  Construct $W$ using
  \begin{align*}
    w_{1} & = \cK(\sigma_{k})^{-\herm}\cC(\sigma_{k})^{\herm},\\
    w_{j} & = \cK(\sigma_{k-j+1})^{-\herm}\cN(\sigma_{k-j+1})^{\herm}w_{j-1},
      & 2 \leq j \leq k,\\
    \mathrm{span}(W) & \supseteq \mathrm{span}\left([w_{1}, \ldots, 
      w_{k}]\right),
  \end{align*}
  and let $V$ be an arbitrary full-rank truncation matrix of appropriate 
  dimension.
  Then the transfer functions of $\hG$ interpolate the transfer functions of 
  $G$ in the following way:
  \begin{align*}
    \begin{aligned}
      G_{1}(\sigma_{k}) & =  \widehat{G}_{1}(\sigma_{k}), &
      G_{2}(\sigma_{k-1},\sigma_{k}) &
        = \widehat{G}_{2}(\sigma_{k-1}, \sigma_{k}), &
      & \ldots, &
      G_{k}(\sigma_{1}, \ldots, \sigma_{k}) &
        = \widehat{G}_{k}(\sigma_{1}, \ldots, \sigma_{k}).
    \end{aligned}
  \end{align*}
\end{theorem}
\begin{proof}
  The proof of this theorem follows analogous to the proof of \Cref{thm:sisov}.
  We only need to note that the left projection space $\mspan(W)$ involves
  the $\cC(s)$ matrix, which takes always the last argument of $G_{k}$ into
  account.
  Therefore, the order of the interpolation points is reversed and the
  recursion formula follows the transfer function order going from left
  to right.
  The rest follows as in the proof of \Cref{thm:sisov} by taking the
  Hermitian conjugate of the matrix functions for the construction.
\end{proof}

The main difference between \Cref{thm:sisov} and
\Cref{thm:sisow} is the order in which the interpolation points
have to be used.
Switching between the two projection schemes leads to a reverse ordering of
the interpolation points for the intermediate transfer functions.

The last theorem of this section states now the combination of 
\Cref{thm:sisov} and \Cref{thm:sisow} by two-sided projection.

\begin{theorem}[Interpolation by two-sided projection]%
  \label{thm:sisovw}
  Let $G$ and $\hG$ be as in \Cref{thm:sisov} and let $V$ be
  constructed as in \Cref{thm:sisov} for a given set of interpolation
  points $\sigma_{1}, \ldots, \sigma_{k} \in \C$ and $W$ as in
  \Cref{thm:sisow} for another set of interpolation points
  $\varsigma_{1}, \ldots, \varsigma_{\theta} \in \C$, for which the matrix 
  functions $\cC(s)$, $\cK(s)^{-1}$, $\cN(s)$ and $\cB(s)$ are defined and 
  $\hcK(s)$ is full-rank.
  Then the transfer functions of $\hG$ interpolate the transfer functions of
  $G$ in the following way:
  \begin{align} \label{prevresult}
    \begin{aligned}
      G_{1}(\sigma_{1}) & = \hG_{1}(\sigma_{1}), &
      G_{2}(\sigma_{1},\sigma_{2}) & = \hG_{2}(\sigma_{1}, \sigma_{2}), &
      & \ldots, & 
      G_{k}(\sigma_{1}, \ldots, \sigma_{k}) &
        = \hG_{k}(\sigma_{1}, \ldots, \sigma_{k}),\\
      G_{1}(\varsigma_{\theta}) & =  \widehat{G}_{1}(\varsigma_{\theta}), &
      G_{2}(\varsigma_{\theta-1},\varsigma_{\theta}) &
        = \widehat{G}_{2}(\varsigma_{\theta-1}, \varsigma_{\theta}), &
      & \ldots, &
      G_{\theta}(\varsigma_{1}, \ldots, \varsigma_{\theta}) &
        = \widehat{G}_{\theta}(\varsigma_{1}, \ldots, \varsigma_{\theta}),
    \end{aligned}
  \end{align}
  and additionally,
  \begin{align} \label{newresult}
    G_{q + \eta}(\sigma_{1}, \ldots, \sigma_{q},
      \varsigma_{\theta-\eta+1}, \ldots, \varsigma_{\theta}) & =
      \hG_{q + \eta}(\sigma_{1}, \ldots, \sigma_{q},
      \varsigma_{\theta-\eta+1}, \ldots, \varsigma_{\theta}),
  \end{align}
  for $1 \leq q \leq k$ and  $1 \leq \eta \leq \theta$.
\end{theorem}
\begin{proof}
  Since the interpolation conditions in~\cref{prevresult} follow directly from
  \Cref{thm:sisov} and \Cref{thm:sisow}, we only need to prove~\cref{newresult}, 
  the mixed interpolation conditions.
  For $q$ and $\eta$ as described in the theorem, we obtain
  \begin{align*}
    & \hG_{q + \eta}(\sigma_{1}, \ldots, \sigma_{q},
      \varsigma_{\theta-\eta+1}, \ldots, \varsigma_{\theta})\\
    & = \hcC(\varsigma_{\theta})\hcK(\varsigma_{\theta})^{-1} \left(
      \prod\limits_{j = 1}^{\eta-1}\hcN(\varsigma_{\theta-j})
      \hcK(\varsigma_{\theta-j})^{-1} \right)
      \left(\prod\limits_{i = 0}^{q-1}\hcN(\sigma_{q-i})
      \hcK(\sigma_{q-i})^{-1} \right) \hcB(\sigma_{1})\\
    & = \hcC(\varsigma_{\theta})\hcK(\varsigma_{\theta})^{-1} \left(
      \prod\limits_{j = 1}^{\eta-1}\hcN(\varsigma_{\theta-j})
      \hcK(\varsigma_{\theta-j})^{-1} \right) W^{\herm}
      \underbrace{\left(\prod\limits_{i = 0}^{q-1}\cN(\sigma_{q-i})
      \cK(\sigma_{q-i})^{-1} \right) \cB(\sigma_{1})}_{
      \phantom{\, \mspan(V)} \in\, \mspan(V)}\\
    & = \underbrace{\cC(\varsigma_{\theta})\cK(\varsigma_{\theta})^{-1} \left(
      \prod\limits_{j = 1}^{\eta-1}\cN(\varsigma_{\theta-j})
      \cK(\varsigma_{\theta-j})^{-1} \right)}_{
      \phantom{\,h,~h^{\herm}\, \in\, \mspan(W)}
      =:\, h,~h^{\herm}\, \in\, \mspan(W)}
      \left(\prod\limits_{i = 0}^{q-1}\cN(\sigma_{q-i})
      \cK(\sigma_{q-i})^{-1} \right) \cB(\sigma_{1})\\
    & = G_{q+\eta}(\sigma_{1}, \ldots, \sigma_{q},
      \varsigma_{\theta-\eta+1}, \ldots, \varsigma_{\theta}),
  \end{align*}
  where we used the construction of $\mspan(V)$ in the third and of $\mspan(W)$ 
  in the fourth lines as denoted and following the  strategy in the proof of 
  \Cref{thm:sisov}.
\end{proof}

It is an important observation that we can interpolate higher-order transfer
functions by only evaluating lower ones for the construction of the model 
reduction bases.
Following \Cref{thm:sisovw}, we can in fact interpolate transfer
functions up to order $k + \theta$.
Also, we recognize that the two-sided projection-based interpolation is able
to match $k + \theta + k \cdot \theta$ interpolation conditions at the same
time.
Those results are similar to the unstructured systems case~\cite{morAntBG20}.
The special case of identical sets of interpolation points is discussed in
the following section regarding Hermite interpolation.


\subsection{Hermite interpolation}

As in the linear case, we can use the projection framework to interpolate not
only the transfer functions but also their derivatives.
In the setting of the multivariate transfer function appearing in bilinear 
systems, this amounts to partial derivatives with respect to the different 
frequency arguments.
For ease of notation, we introduce an abbreviation for partial derivatives
\begin{align*}
  \partial_{s_{1}^{j_{1}} \cdots s_{k}^{j_{k}}} f(z_{1}, \ldots, z_{k}) & :=
    \frac{\partial^{j_{1} + \ldots + j_{k}} f}{\partial s_{1}^{j_{1}} \cdots
    \partial s_{k}^{j_{k}}} (z_{1}, \ldots, z_{k}),
\end{align*}
denoting the differentiation of an analytic function $f\colon \C^{k}
\rightarrow \C^{\ell}$ with respect to the variables $s_{1}, \ldots,
s_{k}$ and evaluated at $z_{1}, \ldots, z_{k} \in \C$.
Moreover, the Jacobian of $f$ is denoted by 
\begin{align*}
  \nabla f & = \begin{bmatrix} \partial_{s_{1}} f & \ldots & \partial_{s_{k}} f 
    \end{bmatrix}
\end{align*}
as the concatenation of all partial derivatives.

The following theorem states a Hermite interpolation result via $V$ only.

\begin{theorem}[Hermite interpolation via $V$]%
  \label{thm:sisovhermite}
  Let $G$ be a bilinear SISO system, described by~\cref{eqn:sisotf},
  and $\hG$ the reduced-order bilinear SISO system constructed
  by~\cref{eqn:proj}.
  Let $\sigma_{1},\ldots,\sigma_{k} \in \C$ be the interpolation points for which
  the matrix functions $\cC(s)$, $\cK(s)^{-1}$, $\cN(s)$ and $\cB(s)$ are 
  analytic and $\hcK(s)$ is full-rank.
  Construct $V$ using
  \begin{align*}
    v_{1, j_{1}} & = \partial_{s^{j_{1}}} (\cK^{-1} \cB) (\sigma_{1}),
      & j_{1} & = 0,\ldots,\ell_{1},\\
    v_{2, j_{2}} & = \partial_{s^{j_{2}}} \cK^{-1} (\sigma_{2})
      \partial_{s^{\ell_{1}}} (\cN \cK^{-1} \cB) (\sigma_{1}),
      & j_{2} & = 0,\ldots,\ell_{2},\\
    & \,\,\,\vdots\\
    v_{k, j_{k}} & = \partial_{s^{j_{k}}} \cK^{-1} (\sigma_{k})
      \left( \prod\limits_{j = 1}^{k-2} \partial_{s^{\ell_{k-j}}} (\cN \cK^{-1})
      (\sigma_{k-j}) \right) \partial_{s^{\ell_{1}}} (\cN \cK^{-1} \cB)
      (\sigma_{1}),
      & j_{k} & = 0,\ldots,\ell_{k},\\
    \mspan(V) & \supseteq \mspan([v_{1,0}, \ldots, v_{k, \ell_{k}}]),
  \end{align*}
  and let $W$ be an arbitrary full-rank truncation matrix of appropriate 
  dimension.
  Then the transfer functions of $\hG$ interpolate the transfer functions of 
  $G$ in the following way:
  \begin{align*}
    \partial_{s_{1}^{j_{1}}} G_{1} (\sigma_{1})
      & = \partial_{s_{1}^{j_{1}}} \hG_{1} (\sigma_{1}),
      & j_{1} & = 0,\ldots,\ell_{1},\\
    \partial_{s_{1}^{\ell_{1}}s_{2}^{j_{2}}} G_{2} (\sigma_{1}, \sigma_{2})
      & = \partial_{s_{1}^{\ell_{1}}s_{2}^{j_{2}}} \hG_{2}
      (\sigma_{1}, \sigma_{2}),
      & j_{2} & = 0,\ldots,\ell_{2},\\
    & \,\,\,\vdots\\
    \partial_{s_{1}^{\ell_{1}} \cdots s_{k-1}^{\ell_{k-1}}s_{k}^{j_{k}}}
      G_{k} (\sigma_{1}, \ldots, \sigma_{k})
      & = \partial_{s_{1}^{\ell_{1}} \cdots s_{k-1}^{\ell_{k-1}}s_{k}^{j_{k}}}
      \hG_{k} (\sigma_{1}, \ldots, \sigma_{k}),
      & j_{k} & = 0,\ldots,\ell_{k}.
  \end{align*}
\end{theorem}
\begin{proof}
  \allowdisplaybreaks
  First, we note that the case $k = 1$ was already proven in~\cite{morBeaG09}
  and $\ell_{1} = \ldots = \ell_{k} = 0$ corresponds to \Cref{thm:sisov}.
  For $k = 2$, we start with $j_{2} = 0$ to investigate the partial derivative
  with respect to $s_{1}$ involving the bilinear term.
  Using the product rule, the partial derivative can be written as
  \begin{align*}
    \partial_{s^{\ell_{1}}} (\cN \cK^{-1} \cB) (\sigma_{1})
      & = \left( \sum\limits_{i_{1} = 0}^{\ell_{1}} c_{i_{1}}
      \partial_{s^{i_{1}}} \cN (\sigma_{1}) \right)
      \left( \sum\limits_{i_{2} = 0}^{\ell_{1}} c_{i_{2}}
      \partial_{s^{i_{2}}} (\cK^{-1} \cB) (\sigma_{1}) \right),
  \end{align*}
  for some appropriate constants $c_{i_{1}}, c_{i_{2}} \in \C$, and
  $i_{1}, i_{2} = 0, \ldots, \ell_{1}$.
  Now, we can show
  \begin{align*}
    \partial_{s_{1}^{\ell_{1}}} \hG_{2} (\sigma_{1},\sigma_{2})
      & = \hcC(\sigma_{2}) \hcK(\sigma_{2})^{-1}
      \partial_{s^{\ell_{1}}} (\hcN \hcK^{-1} \hcB) (\sigma_{1})\\
    & = \hcC(\sigma_{2}) \hcK(\sigma_{2})^{-1}
      \left( \sum\limits_{i_{1} = 0}^{\ell_{1}} c_{i_{1}}
      \partial_{s^{i_{1}}} \hcN (\sigma_{1}) \right)
      \left( \sum\limits_{i_{2} = 0}^{\ell_{1}} c_{i_{2}}
      \partial_{s^{i_{2}}} (\hcK^{-1} \hcB) (\sigma_{1}) \right)\\
    & = \hcC(\sigma_{2}) \hcK(\sigma_{2})^{-1} W^{\herm}
      \left( \sum\limits_{i_{1} = 0}^{\ell_{1}} c_{i_{1}}
      \partial_{s^{i_{1}}} \cN (\sigma_{1}) \right) \\
    & \quad{}\times{} V
      \left( \sum\limits_{i_{2} = 0}^{\ell_{1}} c_{i_{2}}
      \partial_{s^{i_{2}}} ((W^{\herm} \cK V)^{-1} W^{\herm} \cB) (\sigma_{1})
      \right)\\
    & = \hcC(\sigma_{2}) \hcK(\sigma_{2})^{-1} W^{\herm}
      \left( \sum\limits_{i_{1} = 0}^{\ell_{1}} c_{i_{1}}
      \partial_{s^{i_{1}}} \cN (\sigma_{1}) \right)
      \left( \sum\limits_{i_{2} = 0}^{\ell_{1}} c_{i_{2}}
      \partial_{s^{i_{2}}} (\cK^{-1} \cB) (\sigma_{1}) \right)\\
    & = \hcC(\sigma_{2}) \hcK(\sigma_{2})^{-1} W^{\herm}
      \partial_{s^{\ell_{1}}} (\cN \cK^{-1} \cB) (\sigma_{1})\\
    & = \cC(\sigma_{2}) \underbrace{V( W^{\herm} \cK(\sigma_{2}) V)^{-1}
      W^{\herm} \cK(\sigma_{2})}_{\phantom{\, P_{v_{2}}} =:\, P_{v_{2, 0}}} 
      \cK(\sigma_{2})^{-1} \partial_{s^{\ell_{1}}}
      (\cN \cK^{-1} \cB) (\sigma_{1})\\
    & = \cC(\sigma_{2}) \cK(\sigma_{2})^{-1}
      \partial_{s^{\ell_{1}}} (\cN \cK^{-1} \cB) (\sigma_{1})\\
    & = \partial_{s_{1}^{\ell_{1}}} G_{2} (\sigma_{1},\sigma_{2}),
  \end{align*}
  where we first used the construction of $v_{1, j_{1}}$ and then that of
  $v_{2, 0}$ with the projector $P_{v_{2, 0}}$ onto $\mspan(V)$.
  By induction over $j_{2}$, the results for the case $k = 2$ follow
  from~\cite{morBeaG09}; and by induction over $k$ and $j_{k}$, using the same
  arguments, the rest of the theorem follows.
\end{proof}

We note the difference between \Cref{thm:sisov} and
\Cref{thm:sisovhermite} in terms of the subspace construction.
While for the previous interpolation results, we are able to recursively
construct the next part of the model reduction subspace by using the previous 
one, this is not possible in \Cref{thm:sisovhermite} due to the frequency
dependence of the bilinear term $\cN(s)$.
Also, it follows that for the interpolation of the $\ell$-th derivative,
$\ell = \ell_{1} + \ldots + \ell_{k}$, of the $k$-th transfer function
$G_{k}$ in the interpolation points $\sigma_{1}, \ldots, \sigma_{k}$, the
minimal dimension of the projection space $\mspan(V)$ is given by $\ell + k$.

As before, we can consider the counterpart to \Cref{thm:sisovhermite}.
In addition to reversing the order of interpolation points,  the order
of the derivatives needs to be reverted as well for the Hermite interpolation.

\begin{theorem}[Hermite interpolation via $W$]%
  \label{thm:sisowhermite}
  Let $G$, $\hG$ the original and reduced-order models, respectively, and the 
  interpolation points $\sigma_{1},\ldots,\sigma_{k} \in \C$ be as in 
  \Cref{thm:sisovhermite}.
  Construct $W$ using
  \begin{align*}
    w_{1, j_{k}} & = \partial_{s^{j_{k}}} (\cK^{-\herm} \cC^{\herm})
      (\sigma_{k}),
      & j_{k} & = 0, \ldots, \ell_{k},\\
    w_{2, j_{k-1}} & = \partial_{s^{j_{k-1}}} (\cK^{-\herm} \cN^{\herm})
      (\sigma_{k-1}) \partial_{s^{\ell_{k}}} (\cK^{-\herm} \cC^{\herm})
      (\sigma_{k}),
      & j_{k-1} & = 0, \ldots, \ell_{k-1},\\
    & \,\,\,\vdots\\
    w_{k,j_{1}} & =  \partial_{s^{j_{1}}} (\cK^{-\herm} \cN^{\herm})
      (\sigma_{1}) \left( \prod\limits_{j = 2}^{k-1} \partial_{s^{\ell_{j}}}
      (\cK^{-\herm} \cN^{\herm}) (\sigma_{j}) \right) \partial_{s^{\ell_{k}}}\\
    & \quad{}\times{} (\cK^{-\herm} \cC^{\herm}) (\sigma_{k}),
      & j_{1} & = 0, \ldots, \ell_{1},\\
    \mathrm{span}(W) & \supseteq \mspan([w_{1,0}, \ldots, w_{k, \ell_{k}}]),
  \end{align*}
  and let $V$ be an arbitrary full-rank truncation matrix of appropriate 
  dimension.
  Then the transfer functions of $\hG$ interpolate the transfer functions of 
  $G$ in the following way
  \begin{align*}
    \partial_{s_{1}^{j_{k}}} G_{1} (\sigma_{k})
      & = \partial_{s_{1}^{j_{k}}} \hG_{1} (\sigma_{k}),
      & j_{k} & = 0, \ldots, \ell_{k},\\
    \partial_{s_{1}^{j_{k-1}} s_{2}^{\ell_{k}}} G_{2} (\sigma_{k-1}, \sigma_{k})
      & = \partial_{s_{1}^{j_{k-1}} s_{2}^{\ell_{k}}} \hG_{2}
      (\sigma_{k-1}, \sigma_{k}),
      & j_{k-1} & = 0, \ldots, \ell_{k-1},\\
    & \,\,\,\vdots\\
    \partial_{s_{1}^{j_{1}} s_{2}^{\ell_{2}} \cdots s_{k}^{\ell_{k}}}
      G_{k} (\sigma_{1}, \ldots, \sigma_{k})
    & = \partial_{s_{1}^{j_{1}} s_{2}^{\ell_{2}} \cdots s_{k}^{\ell_{k}}}
      \hG_{k} (\sigma_{1}, \ldots, \sigma_{k}),
      & j_{1} & = 0, \ldots, \ell_{1}.
  \end{align*}
\end{theorem}
\begin{proof}
  Observing that the order of the derivatives changed in the same way as the
  interpolation points, the proof works analogously to the proof of
  \Cref{thm:sisovhermite} while building on the ideas from the proof of
  \Cref{thm:sisow}.
\end{proof}

An interesting fact in the structured linear case, as stated  
in~\cite{morBeaG09}, is the implicit matching of Hermite interpolation 
conditions without sampling the derivatives of the transfer function. Next, we 
extend  this construction to the structured bilinear case.
This result becomes a special case of \Cref{thm:sisovw} by using identical sets 
of interpolation points for $V$ and $W$.

\begin{theorem}[Implicit Hermite interpolation by two-sided projection]%
  \label{thm:sisovwhermite}
  Let $G$ and $\hG$ be as in \Cref{thm:sisovhermite}. Also let $V$ and $W$ 
  be constructed as in \Cref{thm:sisov,thm:sisow}, 
  respectively, for the same set of interpolation points $\sigma_{1}, \ldots, 
  \sigma_{k} \in \C$, for which the matrix functions $\cC(s)$, $\cK(s)^{-1}$, 
  $\cN(s)$ and $\cB(s)$ are analytic and $\hcK(s)$ is full-rank.
  Then the transfer functions of $\hG$ interpolate the transfer functions of $G$
  in the following way:
  \begin{align*}
    \begin{aligned}
      G_{1}(\sigma_{1}) & = \hG_{1}(\sigma_{1}), &
      & \ldots, & 
      G_{k-1}(\sigma_{1}, \ldots, \sigma_{k-1}) &
        = \hG_{k-1}(\sigma_{1}, \ldots, \sigma_{k-1}),\\
      G_{1}(\sigma_{k}) & =  \widehat{G}_{1}(\sigma_{k}), &
      & \ldots, &
      G_{k_-1}(\sigma_{2}, \ldots, \sigma_{k}) &
        = \widehat{G}_{k-1}(\sigma_{2}, \ldots, \sigma_{k}),
    \end{aligned}
  \end{align*}
  and additionally,
  \begin{align*}
    G_{k}(\sigma_{1}, \ldots, \sigma_{k})
      & = \hG_{k}(\sigma_{1}, \ldots, \sigma_{k}),\\
    \nabla G_{k}(\sigma_{1}, \ldots, \sigma_{k})
      & = \nabla \hG_{k}(\sigma_{1}, \ldots, \sigma_{k}),\\
    G_{q+\eta}(\sigma_{1}, \ldots, \sigma_{q},
      \sigma_{k-\eta+1}, \ldots, \sigma_{k}) & =
      \hG_{q+\eta}(\sigma_{1}, \ldots, \sigma_{q},
      \sigma_{k-\eta+1}, \ldots, \sigma_{k}),
  \end{align*}
  hold for $1 \leq q, \eta \leq k$.
\end{theorem}
\begin{proof} \allowdisplaybreaks
  While most of the results directly follow from \Cref{thm:sisovw}
  by using identical sets of interpolation points for $V$ and $W$, the Hermite 
  interpolation of the complete Jacobian $\nabla G_{k}$ of the $k$-th 
  order transfer function is new.
  Since $k = 1$ (the linear subsystem) is covered by~\cite{morBeaG09}, we  
  assume $k > 1$.
  Therefore, and by the structure of the multivariate transfer functions $G_{k}$,
  three different cases can occur depending on the differentiation variable,
  i.e., we have
  \begin{align*}
    \begin{aligned}
      \partial_{s_{1}}: && \partial_{s} (\cN \cK^{-1} \cB)
        & = \left( \partial_{s} \cN \right) \cK^{-1} \cB +
        \cN \left( \partial_{s} (\cK^{-1} \cB) \right),\\
      \partial_{s_{j}}: && \partial_{s} (\cN \cK^{-1})
        & = \left( \partial_{s} \cN \right) \cK^{-1} +
        \cN \left( \partial_{s} \cK^{-1} \right), & \text{for}~1 < j < k,\\
      \partial_{s_{k}}: && \partial_{s} (\cC \cK^{-1})
        & = \left( \partial_{s} \cC \right) \cK^{-1} + 
        \cC \left( \partial_{s} \cK^{-1} \right),
    \end{aligned}
  \end{align*}
  as possible derivative terms.
  Since those three cases work analogously to each other, we restrict ourselves,
  for the sake of compactness, to the first one.
  First, we extend the expression of the partial derivative further into
  \begin{align*}
    \partial_{s} (\cN \cK^{-1} \cB) & = \left( \partial_{s} \cN \right)
      \cK^{-1} \cB + \cN \left(-\cK^{-1} \left( \partial_{s} \cK \right)
      \cK^{-1} \cB + \cK^{-1} \left( \partial_{s} \cB \right) \right).
  \end{align*}
  Therefore,  for the complete partial derivative, we obtain
  \begin{align*}
    & \partial_{s_{1}} \hG_{k} (\sigma_{1}, \ldots, \sigma_{k})\\
    & = \hcC(\sigma_{k}) \hcK(\sigma_{k})^{-1} \left(
      \prod\limits_{j = 1}^{k-2} \hcN(\sigma_{k-j}) \hcK(\sigma_{k-j})^{-1} 
      \right) \partial_{s} (\hcN \hcK^{-1} \hcB) (\sigma_{1})\\
    & = \hcC(\sigma_{k}) \hcK(\sigma_{k})^{-1} \left(
      \prod\limits_{j = 1}^{k-2} \hcN(\sigma_{k-j})\hcK(\sigma_{k-j})^{-1} 
      \right)\\
    & \quad{}\times{} \left[ \left( \partial_{s} \hcN \right)
      \hcK^{-1} \hcB - \hcN \cK^{-1} \left( \partial_{s} \hcK \right)
      \hcK^{-1} \hcB + \hcN \hcK^{-1} \left( \partial_{s} \hcB \right) \right]
      (\sigma_{1})\\
    & = \hcC(\sigma_{k}) \hcK(\sigma_{k})^{-1} \left(
      \prod\limits_{j = 1}^{k-2} \hcN(\sigma_{k-j}) \hcK(\sigma_{k-j})^{-1} 
      \right) \partial_{s} \hcN (\sigma_{1}) \hcK(\sigma_{1})^{-1}
      \hcB(\sigma_{1})\\
    & \quad{}-{} \hcC(\sigma_{k}) \hcK(\sigma_{k})^{-1} \left(
      \prod\limits_{j = 1}^{k-2} \hcN(\sigma_{k-j}) \hcK(\sigma_{k-j})^{-1} 
      \right) \hcN(\sigma_{1}) \hcK(\sigma_{1})^{-1}
      \partial_{s} \hcK (\sigma_{1})
      \hcK(\sigma_{1})^{-1} \hcB(\sigma_{1})\\
    & \quad{}+{} \hcC(\sigma_{k}) \hcK(\sigma_{k})^{-1} \left(
      \prod\limits_{j = 1}^{k-2} \hcN(\sigma_{k-j}) \hcK(\sigma_{k-j})^{-1} 
      \right) \hcN(\sigma_{1}) \hcK(\sigma_{1})^{-1}
      \partial_{s} \hcB (\sigma_{1})\\
    & = \underbrace{\cC(\sigma_{k}) \cK(\sigma_{k})^{-1} \left(
      \prod\limits_{j = 1}^{k-2} \cN(\sigma_{k-j}) \cK(\sigma_{k-j})^{-1} 
      \right)}_{\phantom{\, h_{1},~h_{1}^{\herm}\, \in\, \mspan(W)}
      =:\, h_{1},~h_{1}^{\herm}\, \in\, \mspan(W)}
      \partial_{s} \cN (\sigma_{1})
      \underbrace{\cK(\sigma_{1})^{-1} \cB(\sigma_{1})
      \vphantom{\left(\prod\limits_{j = 1}^{k-2}\right)}}_{
      \phantom{\, \mspan(V)} \in\, \mspan(V)}\\
    & \quad{}-{} \underbrace{\cC(\sigma_{k}) \cK(\sigma_{k})^{-1} \left(
      \prod\limits_{j = 1}^{k-2} \cN(\sigma_{k-j}) \cK(\sigma_{k-j})^{-1} 
      \right)\cN(\sigma_{1}) \cK(\sigma_{1})^{-1}}_{
      \phantom{\, h_{2},~h_{2}^{\herm}\, \in\,\mspan(W)}
      =:\, h_{2},~h_{2}^{\herm}\, \in\,\mspan(W)}
      \partial_{s} \cK (\sigma_{1})
      \underbrace{\cK(\sigma_{1})^{-1} \cB(\sigma_{1})
      \vphantom{\left(\prod\limits_{j = 1}^{k-2}\right)}}_{
      \phantom{\, \mspan(V)} \in\, \mspan(V)}\\
    & \quad{}+{} \underbrace{\cC(\sigma_{k}) \cK(\sigma_{k})^{-1} \left(
      \prod\limits_{j = 1}^{k-2} \cN(\sigma_{k-j}) \cK(\sigma_{k-j})^{-1} 
      \right) \cN(\sigma_{1}) \cK(\sigma_{1})^{-1}}_{
      \phantom{\, h_{2},~h_{2}^{\herm}\, \in\, \mspan(W)}
      =\, h_{2},~h_{2}^{\herm}\, \in\, \mspan(W)}
      \partial_{s} \cB (\sigma_{1})\\
    & = \partial_{s_{1}} G_{k} (\sigma_{1}, \ldots, \sigma_{k}),
  \end{align*}
  where we used, as denoted by the underbraces, the construction of
  either $\mspan(W)$ or $\mspan(V)$, and the fact that the model reduction bases 
  $V$ and $W$ are constant matrices.
  As stated before,  the results for the other partial derivatives follow 
  analogously, which proves interpolation of the full Jacobian in the end.
\end{proof}

As in the previous section, by using two-sided projection we can match 
interpolation conditions for a larger number of interpolation points and 
higher-order transfer functions.
Following the results of \Cref{thm:sisovw} we can expect, using
derivatives for the two-sided projection, to match at least
$(k + \ell) + (\theta + \nu) + (k + \ell) \cdot (\theta +
\nu)$ transfer function values, where $k, \ell$ relate to $\mspan(V)$ and
$\theta, \nu$ to $\mspan(W)$, and where $\ell = \ell_{1} + \ldots + \ell_{k}$
and $\nu = \nu_{1} + \ldots + \nu_{\theta}$ denote the orders of the
partial derivatives and $k, \theta$ the orders of the transfer functions to
interpolate.

\begin{theorem}[Hermite interpolation by two-sided projection]%
  \label{thm:sisovwderivative}
  Let $G$ and $\hG$ be as in \Cref{thm:sisovhermite} and let $V$
  be constructed as in \Cref{thm:sisovhermite} for a given set of
  interpolation points $\sigma_{1}, \ldots, \sigma_{k} \in \C$ and orders of
  partial derivatives $\ell_{1}, \ldots, \ell_{k}$, and $W$ as in
  \Cref{thm:sisowhermite} for another set of interpolation points
  $\varsigma_{1}, \ldots, \varsigma_{\theta} \in \C$ and orders of partial 
  derivatives $\nu_{1}, \ldots, \nu_{\theta}$, for which the matrix functions
  $\cC(s), \cK(s)^{-1}, \cN(s)$ and $\cB(s)$ are analytic and $\hcK(s)$ has
  full-rank.
  Then the transfer functions of $\hG$ interpolate the transfer functions of $G$
  in the following way:
  \begin{align*}
    \partial_{s_{1}^{j_{1}}} G_{1} (\sigma_{1})
      & = \partial_{s_{1}^{j_{1}}} \hG_{1} (\sigma_{1}),
      & j_{1} & = 0,\ldots,\ell_{1},\\
    & \,\,\,\vdots\\
    \partial_{s_{1}^{\ell_{1}} \cdots s_{k-1}^{\ell_{k-1}} s_{k}^{j_{k}}}
      G_{k} (\sigma_{1}, \ldots, \sigma_{k})
      & = \partial_{s_{1}^{\ell_{1}} \cdots s_{k-1}^{\ell_{k-1}} s_{k}^{j_{k}}}
      \hG_{k} (\sigma_{1}, \ldots, \sigma_{k}),
      & j_{k} & = 0,\ldots,\ell_{k},\\
    \partial_{s_{1}^{i_{\theta}}} G_{1} (\varsigma_{\theta})
      & = \partial_{s_{1}^{i_{\theta}}} \hG_{1} (\varsigma_{\theta}),
      & i_{\theta} & = 0,\ldots,\nu_{\theta},\\
    & \,\,\,\vdots\\
    \partial_{s_{1}^{i_{1}} s_{2}^{\nu_{2}} \cdots s_{\theta}^{\nu_{\theta}}}
      G_{\theta} (\varsigma_{1}, \ldots, \varsigma_{\theta})
      & = \partial_{s_{1}^{i_{1}} s_{2}^{\nu_{2}} \cdots
      s_{\theta}^{\nu_{\theta}}} \hG_{\theta} (\varsigma_{1}, \ldots,
      \varsigma_{\theta}),
      & i_{1} & = 0,\ldots,\nu_{1},
  \end{align*}
  and additionally,
  \begin{align*}
    & \partial_{s_{1}^{\ell_{1}} \cdots s_{q-1}^{\ell_{q-1}} s_{q}^{j_{q}}
      s_{q+1}^{i_{\theta - \eta + 1}} s_{q+2}^{\nu_{\theta - \eta + 2}}
      \cdots s_{q + \eta}^{\nu_{\theta}}} G_{q + \eta} (\sigma_{1}, \ldots,
      \sigma_{q}, \varsigma_{\theta - \eta + 1}, \ldots, \varsigma_{\theta})\\
    & = \partial_{s_{1}^{\ell_{1}} \cdots s_{q-1}^{\ell_{q-1}} s_{q}^{j_{q}}
      s_{q+1}^{i_{\theta - \eta + 1}} s_{q+2}^{\nu_{\theta - \eta + 2}}
      \cdots s_{q + \eta}^{\nu_{\theta}}} \hG_{q + \eta} (\sigma_{1}, \ldots,
      \sigma_{q}, \varsigma_{\theta - \eta + 1}, \ldots, \varsigma_{\theta})
  \end{align*}
  holds for $j_{q} = 0, \ldots, \ell_{q}$; $i_{\theta - \eta + 1} = 0, \ldots,
  \nu_{\theta - \eta + 1}$; $1 \leq q \leq k$ and $1 \leq \eta \leq \theta$.
\end{theorem}
\begin{proof} \allowdisplaybreaks
  As for \Cref{thm:sisovw}, the first parts of the result just summarize
  the theorems stating the one-sided projection approaches
  (\Cref{thm:sisovhermite,thm:sisowhermite}), i.e., we only
  need to prove the additional interpolation constraints with the mixed
  partial derivatives.
  It holds
  \begin{align*}
    & \partial_{s_{1}^{\ell_{1}} \cdots s_{q-1}^{\ell_{q-1}} s_{q}^{j_{q}}
      s_{q+1}^{i_{\theta - \eta + 1}} s_{q+2}^{\nu_{\theta - \eta + 2}}
      \cdots s_{q + \eta}^{\nu_{\theta}}} G_{q + \eta} (\sigma_{1}, \ldots,
      \sigma_{q}, \varsigma_{\theta - \eta + 1}, \ldots, \varsigma_{\theta})\\
    & =  \partial_{s^{\nu_{\theta}}} (\hcC \hcK^{-1}) (\varsigma_{\theta})
      \cdots
      \partial_{s^{\nu_{\theta-\eta+2}}} (\hcN \hcK^{-1})
      (\varsigma_{\theta-\eta+2})
      \partial_{s^{i_{\theta-\eta+1}}} (\hcN \hcK^{-1})
      (\varsigma_{\theta-\eta+1})\\
    & \quad{}\times{} \partial_{s^{j_{q}}} (\hcN \hcK^{-1}) (\sigma_{q})
      \partial_{s^{\ell_{q-1}}} (\hcN \hcK^{-1}) (\sigma_{q-1})
      \cdots
      \partial_{s^{\ell_{1}}} (\hcN \hcK^{-1} \hcB) (\sigma_{1})\\
    & = \underbrace{\partial_{s^{\nu_{\theta}}} (\cC \cK^{-1})
      (\varsigma_{\theta})
      \cdots
      \partial_{s^{\nu_{\theta-\eta+2}}} (\cN \cK^{-1})
      (\varsigma_{\theta-\eta+2})
      \partial_{s^{i_{\theta-\eta+1}}} (\cN \cK^{-1})
      (\varsigma_{\theta-\eta+1})}_{
      \phantom{\, h,~h\, \in\, \mspan(W)} =:\, h,~h\, \in\, \mspan(W)}\\
    & \quad{}\times{} \underbrace{\partial_{s^{j_{q}}} (\cN \cK^{-1})
      (\sigma_{q})
      \partial_{s^{\ell_{q-1}}} (\cN \cK^{-1})
     (\sigma_{q-1})
      \cdots
      \partial_{s^{\ell_{1}}} (\cN \cK^{-1} \cB)
      (\sigma_{1})}_{\phantom{\, \mspan(V)} \in\, \mspan(V)}\\
    & = \partial_{s_{1}^{\ell_{1}} \cdots s_{q-1}^{\ell_{q-1}} s_{q}^{j_{q}}
      s_{q+1}^{i_{\theta - \eta + 1}} s_{q+2}^{\nu_{\theta - \eta + 2}}
      \cdots s_{q + \eta}^{\nu_{\theta}}} \hG_{q + \eta} (\sigma_{1}, \ldots,
      \sigma_{q}, \varsigma_{\theta - \eta + 1}, \ldots, \varsigma_{\theta})
  \end{align*}
  for $j_{q} = 0, \ldots, \ell_{q}$; $i_{\theta - \eta + 1} = 0, \ldots,
  \nu_{\theta - \eta + 1}$; $1 \leq q \leq k$ and $1 \leq \eta \leq \theta$.
\end{proof}

For an easier understanding of \Cref{thm:sisovwderivative}, we
consider here a small theoretical example, where we only interpolate the linear
part choosing $k = \theta = 1$, the interpolation points $\sigma,
\varsigma$ and for the partial derivatives $\ell = \ell_{1} = 2$ and $\nu =
\nu_{1} = 1$.
Then using the first part of \Cref{thm:sisovwderivative} we enforce
interpolation of the following terms by means of $\mspan(V)$:
\begin{align*}
  \begin{aligned}
    G_{1}(\sigma), && \partial_{s_{1}} G_{1} (\sigma), &&
      \partial_{s_{1}^{2}} G_{1} (\sigma),
  \end{aligned}
\end{align*}
And similarly via $\mspan(W)$, we enforce interpolation of
\begin{align*}
  \begin{aligned}
    G_{1}(\varsigma), && \partial_{s_{1}} G_{1} (\varsigma).
  \end{aligned}
\end{align*}
By using two-sided projection, we can now additionally match higher-order
transfer functions and their partial derivatives, namely
\begin{align*}
  \begin{aligned}
    G_{2}(\sigma, \varsigma), &&
      \partial_{s_{1}} G_{2} (\sigma, \varsigma), &&
      \partial_{s_{2}} G_{2} (\sigma, \varsigma), &&
      \partial_{s_{1}^{2}} G_{2} (\sigma, \varsigma), &&
      \partial_{s_{1}s_{2}} G_{2} (\sigma, \varsigma), &&
      \partial_{s_{1}^{2}s_{2}} G_{2} (\sigma, \varsigma).
  \end{aligned}
\end{align*}
As already realized in \Cref{thm:sisovwhermite}, two-sided projection
with the same sets of interpolation points leads to additional interpolation
of derivatives.
This also works in combination with \Cref{thm:sisovwderivative}.
The following corollary states a particular special case.

\begin{corollary} \label{cor:tothm}
  Assume $G$ and $\hG$ are constructed as in \Cref{thm:sisovwderivative}
  for identical sets of interpolation points $\sigma_{1}, \ldots, \sigma_{k} \in 
  \C$ and matching orders of the partial derivatives, i.e.,
  $\ell_{1} = \nu_{1}$, \ldots, $\ell_{k} = \nu_{k}$.
  Then additionally to the interpolation results of
  \Cref{thm:sisovwderivative} it holds
  \begin{align*}
    \nabla \left( \partial_{s_{1}^{\ell_{1}} \cdots s_{k}^{\ell_{k}}} G_{k}
      (\sigma_{1}, \ldots, \sigma_{k}) \right)
      & = \nabla \left( \partial_{s_{1}^{\ell_{1}} \cdots s_{k}^{\ell_{k}}} 
      \hG_{k} (\sigma_{1}, \ldots, \sigma_{k}) \right).
  \end{align*}
\end{corollary}
\begin{proof}
  The proof follows directly from \Cref{thm:sisovwhermite} by setting
  the last partial derivative as the final interpolation condition of the left
  and right projection spaces.
\end{proof}


\subsection{Numerical examples}
\label{sec:sisoexamples}

We illustrate the SISO analysis using two numerical
examples, having the structured bilinearities as in  
\Cref{sec:2ndorderexm,sec:time-delay-models}.
We compare our resulting structure-preserving interpolation
framework to other approaches from the literature that have been used to 
approximate structured bilinear systems \emph{without preserving the structure}, 
as in, e.g.,~\cite{morAlbFB93, morGosPBetal19}.

We compare the approximation error both in time and frequency domains.
In time domain, we display a point-wise relative output error for a given input 
signal, namely
\begin{align*}
  \frac{\lvert y(t) - \hy(t) \rvert}{\lvert y(t) \rvert},
\end{align*}
for $t \in [0, t_{f}]$, and in frequency domain, we display the point-wise 
relative error of the first and second subsystem transfer functions, i.e.,
\begin{align*}
  \begin{aligned}
    \frac{\lvert G_{1}(\omega_{1} \iu) - \hG_{1}(\omega_{1} \iu) \rvert}{\lvert 
      G_{1}(\omega_{1} \iu) \rvert} &&& \text{and} &
      \frac{\lvert G_{2}(\omega_{1} \iu, \omega_{2} \iu) -
      \hG_{2}(\omega_{1} \iu, \omega_{2} \iu) \rvert}{\lvert G_{2}(\omega_{1} 
      \iu, \omega_{2} \iu) \rvert},
  \end{aligned}
\end{align*}
for the frequencies $\omega_{1}, \omega_{2} \in [\omega_{\min}, \omega_{\max}]$.

The experiments reported here have been executed on a machine with 2 Intel(R)
Xeon(R) Silver 4110 CPU processors running at 2.10\,GHz and equipped with
192\,GB total main memory.
The computer runs on CentOS Linux release 7.5.1804 (Core) using
MATLAB 9.7.0.1190202 (R2019b).


\subsubsection{Damped mass-spring system}%
\label{sec:siso_msd}

First, we consider a damped mass-spring system.
The linear parts of the dynamics are modeled as in~\cite{morMehS05}, describing 
a chain of masses connected by springs and dampers, where each mass is 
additionally connected to a separate spring and damper.
In order to  focus on only the mechanical structure, we removed the holonomic
constraint from~\cite{morMehS05}.
For the bilinear part, the springs are modeled to be dependent on the applied
external force, such that a displacement to the right increases the stiffness
due to compression of the springs and to the left decreases it due to the
appearing strain.
This results in a structured bilinear control system of the form
\begin{align} \label{eqn:msd}
  \begin{aligned}
    M \ddot{q}(t) + D \dot{q}(t) + K q(t) & = N_{\rm{p}} q(t) u(t)
      + B_{\rm{u}} u(t),\\
    y(t) & = C_{\rm{p}} q(t),
  \end{aligned}
\end{align}
with $M, D, K, N_{\rm{p}} \in \R^{n \times n}$ and
$B_{\rm{u}}, C_{\rm{p}}^{\trans} \in \R^{n}$.
The input matrix is chosen to apply the external force only to the first mass,
i.e., $B = e_{1}$, and the output gives the displacement of the second mass,
i.e., $C = e_{2}^{\trans}$, where $e_{i}$ denotes the $i$-th column of the 
identity matrix $I_{n}$.
The bilinear term is a scaled version of the stiffness matrix
\begin{align*}
  N_{\rm{p}} = -SKS,
\end{align*}
where $S$ is a diagonal matrix containing entries as
\texttt{linspace(0.2, 0, n)}.
For our experiment, we have chosen the original system to consist of
$n = 1\,000$ masses.

We construct three reduced-order models: (i) our structure-preserving bilinear
interpolation, denoted by StrInt, (ii) two unstructured classical bilinear 
approximations by converting~\cref{eqn:msd} to first-order 
form~\cref{eqn:bsosyscomp} followed by interpolatory model reduction of this 
first-order system, denoted by FOInt. Note that  FOInt yields a 
reduced-order model of the form~\cref{eqn:bsosyscomp}, which does not retain the 
underlying physical structure.
Also, it needs to be remarked that the computational effort for the construction 
of FOInt is higher than for the structure-preserving approach due to solving 
underlying linear systems of doubled size, even in a structure exploiting 
implementation; see, e.g.,~\cite{morBenS11}.
Since the original system is a mechanical model, we use only a
one-sided projection to preserve the mechanical properties in the reduced-order
model, i.e., we apply \Cref{thm:sisov} and set $W = V$.
For all approximants, we focus on  the first and second transfer
functions and choose purely imaginary interpolation points.
We construct StrInt and FOInt(12) by using the interpolation points
$\pm \texttt{logspace(-4, 4, 3)} \iu$ such that the resulting reduced-order 
bilinear systems are of order $r = 12$, giving two different interpolations in 
the same frequency points.
Since bilinear second-order systems can be rewritten as first-order systems by 
doubling the state-space dimension, we construct additionally a second 
unstructured approximation FOInt(24) of order $r = 24$ by using
$\pm \texttt{logspace(-4, 4, 6)} \iu$, which has twice the order of StrInt.

\begin{figure}[tb]
  \begin{center}
    \begin{subfigure}[t]{.49\textwidth}
      \begin{center}
  \tikzexternalenable%
  \tikzsetnextfilename{msd_time_sim}%
  \filemodCmp{graphics/msd_time_sim.tikz}{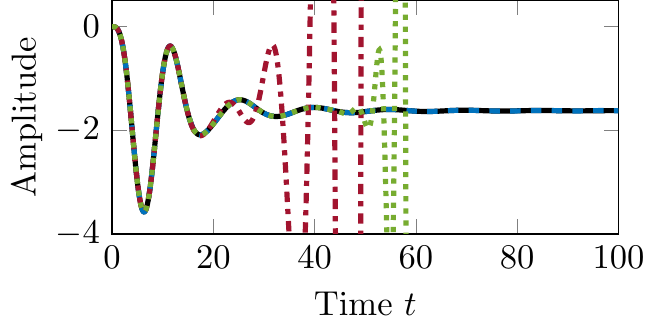}%
    {\tikzset{external/remake next}}{}%
  \begin{tikzpicture}[
  every axis/.append style = {
    scaled x ticks = false,
    x tick label style = {/pgf/number format/fixed},
    cycle list name = plotlist}
  ]
  \pgfplotstableread{graphics/data/msd_time_sim.dat}\tableSIM
  
  \begin{axis}[%
    width  = .7\textwidth,
    height = .1\textheight,
    scale only axis,
    xmin = 0,
    xmax = 100,
    restrict y to domain = -8:3,
    ymin = -4,
    ymax = 0.5,
    xminorticks = false,
    yminorticks = false,
    xlabel = {Time $t$},
    ylabel = {Amplitude},
    ylabel style = {yshift = -.3em}]
        
    \addplot table[x index = 0, y index = 1] {\tableSIM};
    \addplot table[x index = 0, y index = 2] {\tableSIM};
    \addplot table[x index = 0, y index = 3] {\tableSIM};
    \addplot table[x index = 0, y index = 4] {\tableSIM};
  \end{axis}
\end{tikzpicture}%
  \tikzexternaldisable%

        \subcaption{Time response.}
      \end{center}
    \end{subfigure}
    \hfill
    \begin{subfigure}[t]{.49\textwidth}
      \begin{center}
  \tikzexternalenable%
  \tikzsetnextfilename{msd_time_err}%
  \filemodCmp{graphics/msd_time_err.tikz}{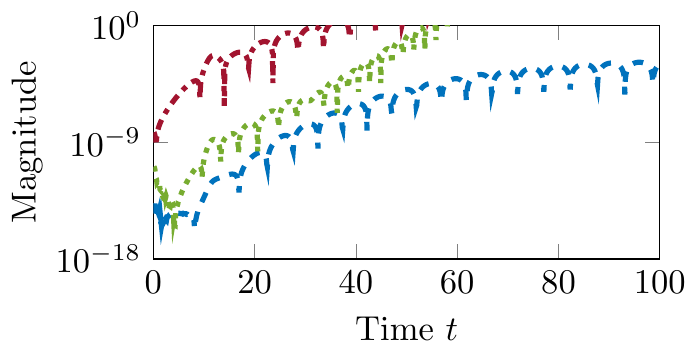}%
    {\tikzset{external/remake next}}{}%
  \begin{tikzpicture}[
  every axis/.append style = {
    scaled x ticks = false,
    x tick label style = {/pgf/number format/fixed},
    cycle list name = plotlist}
  ]
  \pgfplotstableread{graphics/data/msd_time_err.dat}\tableRELERR
  
  \begin{semilogyaxis}[%
    width  = .7\textwidth,
    height = .1\textheight,
    scale only axis,
    xmin = 0,
    xmax = 100,
    ymin = 1e-18,
    ymax = 1,
    xminorticks = false,
    yminorticks = false,
    xlabel = {Time $t$},
    ylabel = {Magnitude},
    ylabel style = {yshift = -.3em}]
    
    \pgfplotsset{cycle list shift = 1}
    \addplot table[x index = 0, y index = 1] {\tableRELERR};
    \addplot table[x index = 0, y index = 2] {\tableRELERR};
    \addplot table[x index = 0, y index = 3] {\tableRELERR};
  \end{semilogyaxis}
\end{tikzpicture}%
  \tikzexternaldisable%

        \subcaption{Relative errors.}
      \end{center}
    \end{subfigure}
    \vspace{.5\baselineskip}

  \tikzexternalenable%
  \tikzsetnextfilename{msd_time_legend}%
  \filemodCmp{graphics/msd_time_legend.tikz}{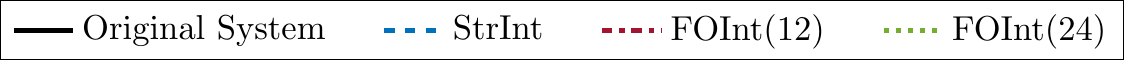}%
    {\tikzset{external/remake next}}{}%
  \begin{tikzpicture}[
  every axis/.append style = {
    scaled x ticks = false,
    x tick label style = {/pgf/number format/fixed},
    cycle list name = plotlist}
  ]
  
  \begin{axis}[%
    hide axis,
    scale only axis,
    width = 1mm,
    legend columns = 4, 
    legend style = {
      at     = {(0,0)},
      anchor = center,
      /tikz/every even column/.append style = {column sep = 0.5cm}}]
    \pgfplotsinvokeforeach{1,...,4}{\addplot coordinates {(0,0)};}
    
    \addlegendentry{Original System};
    \addlegendentry{StrInt};
    \addlegendentry{FOInt(12)};
    \addlegendentry{FOInt(24)};
  \end{axis}
\end{tikzpicture}%
  \tikzexternaldisable%

    \caption{Time simulation results for the damped mass-spring system.}
    \label{fig:msd_time}
  \end{center}
\end{figure}
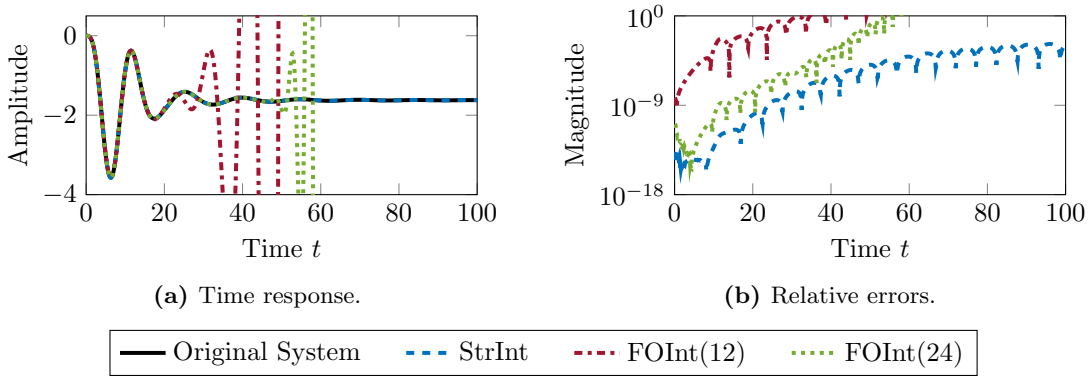

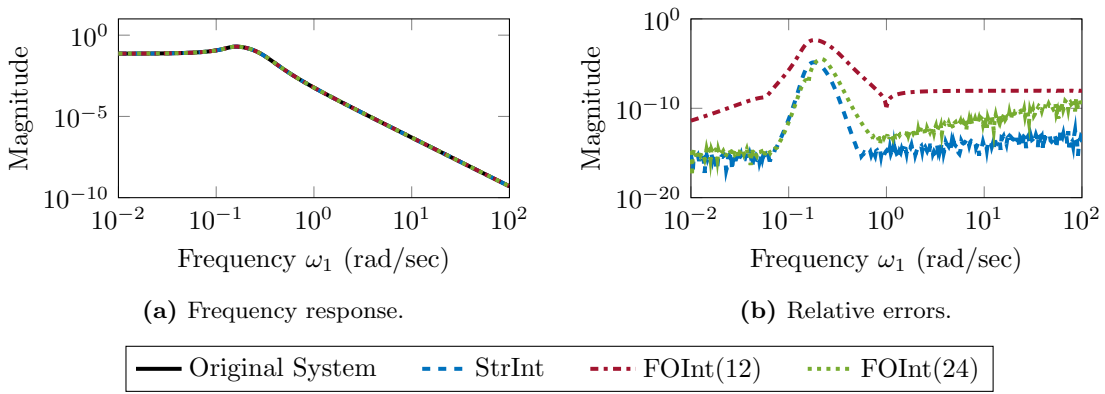
\begin{figure}[tb]
  \begin{center}
    \begin{subfigure}[t]{.49\textwidth}
      \begin{center}
  \tikzexternalenable%
  \tikzsetnextfilename{msd_freq_g1_tf}%
  \filemodCmp{graphics/msd_freq_g1_tf.tikz}{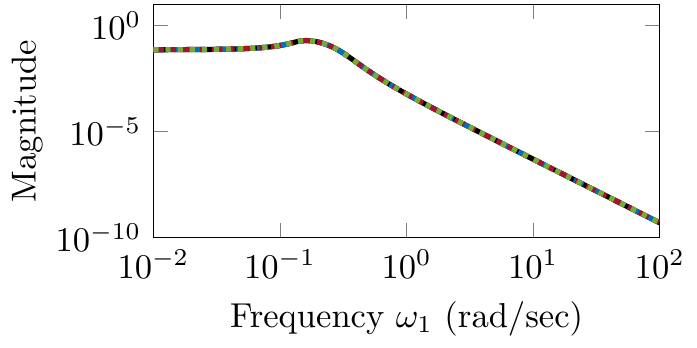}%
    {\tikzset{external/remake next}}{}%
  \begin{tikzpicture}[
  every axis/.append style = {
    scaled x ticks = false,
    x tick label style = {/pgf/number format/fixed},
    cycle list name = plotlist}
  ]
  \pgfplotstableread{graphics/data/msd_freq_g1_tf.dat}\tableTF

  \begin{loglogaxis}[%
    width  = .7\textwidth,
    height = .1\textheight,
    scale only axis,
    xmin = 1e-2,
    xmax = 1e+2,
    ymin = 1e-10,
    ymax = 1e+1,
    xminorticks = false,
    yminorticks = false,
    xlabel = {Frequency $\omega_{1}$ (rad/sec)},
    ylabel = {Magnitude},
    ylabel style = {yshift = -.3em}]
        
    \addplot table[x index = 0, y index = 1] {\tableTF};
    \addplot table[x index = 0, y index = 2] {\tableTF};
    \addplot table[x index = 0, y index = 3] {\tableTF};
    \addplot table[x index = 0, y index = 4] {\tableTF};
  \end{loglogaxis}
\end{tikzpicture}%
  \tikzexternaldisable%

        \subcaption{Frequency response.}
      \end{center}
    \end{subfigure}
    \hfill
    \begin{subfigure}[t]{.49\textwidth}
      \begin{center}
  \tikzexternalenable%
  \tikzsetnextfilename{msd_freq_g1_err}%
  \filemodCmp{graphics/msd_freq_g1_err.tikz}{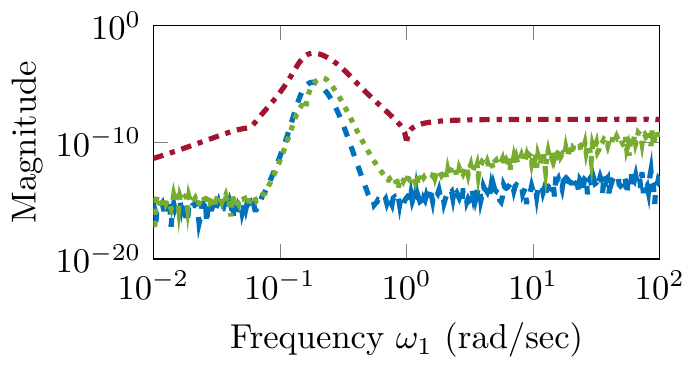}%
    {\tikzset{external/remake next}}{}%
  \begin{tikzpicture}[
  every axis/.append style = {
    scaled x ticks = false,
    x tick label style = {/pgf/number format/fixed},
    cycle list name = plotlist}
  ]
  \pgfplotstableread{graphics/data/msd_freq_g1_err.dat}\tableRELERR
  
  \begin{loglogaxis}[%
    width  = .7\textwidth,
    height = .1\textheight,
    scale only axis,
    xmin = 1e-2,
    xmax = 1e+2,
    ymin = 1e-20,
    ymax = 1,
    xminorticks = false,
    yminorticks = false,
    xlabel={Frequency $\omega_{1}$ (rad/sec)},
    ylabel = {Magnitude},
    ylabel style = {yshift = -.3em}]
        
    \pgfplotsset{cycle list shift = 1}
    \addplot table[x index = 0, y index = 1] {\tableRELERR};
    \addplot table[x index = 0, y index = 2] {\tableRELERR};
    \addplot table[x index = 0, y index = 3] {\tableRELERR};
  \end{loglogaxis}
\end{tikzpicture}%
  \tikzexternaldisable%

        \subcaption{Relative errors.}
      \end{center}
    \end{subfigure}
    \vspace{.5\baselineskip}

  \tikzexternalenable%
  \tikzsetnextfilename{msd_freq_g1_legend}%
  \filemodCmp{graphics/msd_freq_g1_legend.tikz}{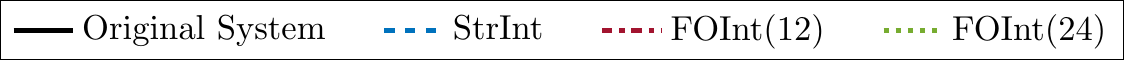}%
    {\tikzset{external/remake next}}{}%
  \begin{tikzpicture}[
  every axis/.append style = {
    scaled x ticks = false,
    x tick label style = {/pgf/number format/fixed},
    cycle list name = plotlist}
  ]
  
  \begin{axis}[%
    hide axis,
    scale only axis,
    width = 1mm,
    legend columns = 4, 
    legend style = {
    at     = {(0,0)},
    anchor = center,
    /tikz/every even column/.append style = {column sep = 0.5cm}}]
    \pgfplotsinvokeforeach{1,...,4}{\addplot coordinates {(0,0)};}
        
    \addlegendentry{Original System};
    \addlegendentry{StrInt};
    \addlegendentry{FOInt(12)};
    \addlegendentry{FOInt(24)};
  \end{axis}
\end{tikzpicture}%
  \tikzexternaldisable%

    \caption{Frequency domain results of the first transfer functions for the
      damped mass-spring system.}
    \label{fig:msd_freq_g1}
  \end{center}
\end{figure}

\Cref{fig:msd_time} shows the time output of the original system,
as well as that of the structure-preserving (StrInt) and first-order 
interpolations (FOInt(12), FOInt(24)), where we applied the input signal
\begin{align*}
  u(t) & = \sin(200 t) + 200,
\end{align*}
which can be seen as a step signal with a sinusoidal disturbance.
We see that while all three outputs are indistinguishable in the 
beginning, FOInt(12) becomes unstable after approximately $20$ time steps and 
FOInt(24) after around $50$ time steps, while StrInt accurately approximates the 
original system over the whole time range of interest.
Even though the linear dynamics in FOInt(12) and FOInt(24) are asymptotically 
stable, these reduced-order models completely lack the underlying physical 
mechanical structure and they become unstable for the chosen input signal. On 
the other hand, by using one-sided projection, StrInt preserves all the 
mechanical (and physical) properties of the original system in terms of symmetry 
and definiteness of the system matrices, which then leads to the stable time 
simulation behavior in this case.
\Cref{fig:msd_freq_g1,fig:msd_freq_g2} show the approximation
results in the frequency domain for the first two transfer functions.
Comparing StrInt and FOInt(12), the structure-preserving approximation is orders
of magnitude better than the unstructured approximation of the same size.
StrInt and FOInt(24) behave mainly the same, while, for higher frequencies, we 
can observe a numerical drift-off of the unstructured approximation.

\begin{figure}[tb]
  \begin{center}
    \begin{subfigure}[t]{.49\textwidth}
      \begin{center}
  \tikzexternalenable%
  \tikzsetnextfilename{msd_freq_g2_err_so}%
  \filemodCmp{graphics/msd_freq_g2_err_so.tikz}{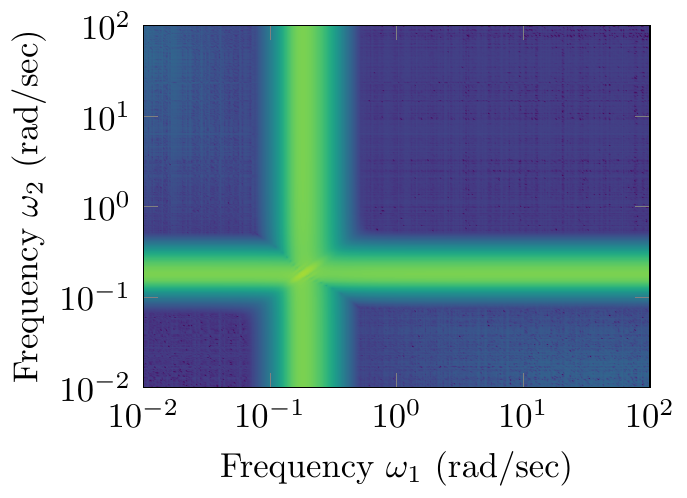}%
    {\tikzset{external/remake next}}{}%
  \begin{tikzpicture}
  \begin{loglogaxis}[
    view   = {0}{90},
    width  = .7\textwidth,
    height = .5\textwidth,
    scale only axis,
    axis on top,
    xmin   = 1e-2,
    xmax   = 1e+2,
    ymin   = 1e-2,
    ymax   = 1e+2,
    xtick  = {1e-2, 1e-1, 1e0, 1e+1, 1e+2},
    ytick  = {1e-2, 1e-1, 1e0, 1e+1, 1e+2},
    xminorticks = false,
    yminorticks = false,
    xlabel = {Frequency $\omega_{1}$ (rad/sec)},
    ylabel = {Frequency $\omega_{2}$ (rad/sec)},
    ylabel style = {yshift = -.3em},
    scaled x ticks = false,
    x tick label style = {/pgf/number format/fixed}]
        
      \addplot graphics[xmin = 1e-2, xmax = 1e+2, ymin = 1e-2, ymax = 1e+2]
        {graphics/data/msd_freq_g2_err_so.pdf};
            
  \end{loglogaxis}
\end{tikzpicture}%
  \tikzexternaldisable%

        \subcaption{StrInt.}
      \end{center}
    \end{subfigure}
    \hfill
    \begin{subfigure}[t]{.49\textwidth}
      \begin{center}
  \tikzexternalenable%
  \tikzsetnextfilename{msd_freq_g2_err_fo1}%
  \filemodCmp{graphics/msd_freq_g2_err_fo1.tikz}{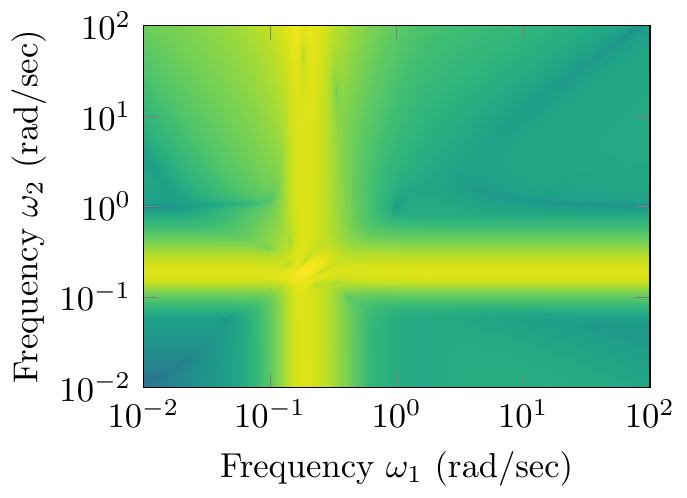}%
    {\tikzset{external/remake next}}{}%
  \begin{tikzpicture}
  \begin{loglogaxis}[
    view   = {0}{90},
    width  = .7\textwidth,
    height = .5\textwidth,
    scale only axis,
    axis on top,
    xmin   = 1e-2,
    xmax   = 1e+2,
    ymin   = 1e-2,
    ymax   = 1e+2,
    xtick  = {1e-2, 1e-1, 1e0, 1e+1, 1e+2},
    ytick  = {1e-2, 1e-1, 1e0, 1e+1, 1e+2},
    xminorticks = false,
    yminorticks = false,
    xlabel = {Frequency $\omega_{1}$ (rad/sec)},
    ylabel = {Frequency $\omega_{2}$ (rad/sec)},
    ylabel style = {yshift = -.3em},
    scaled x ticks = false,
    x tick label style = {/pgf/number format/fixed}]
        
      \addplot graphics[xmin = 1e-2, xmax = 1e+2, ymin = 1e-2, ymax = 1e+2]
        {graphics/data/msd_freq_g2_err_fo1.pdf};
            
  \end{loglogaxis}
\end{tikzpicture}%
  \tikzexternaldisable%

        \subcaption{FOInt(12).}
      \end{center}
    \end{subfigure}
    
    \begin{subfigure}[t]{.49\textwidth}
      \begin{center}
  \tikzexternalenable%
  \tikzsetnextfilename{msd_freq_g2_err_fo2}%
  \filemodCmp{graphics/msd_freq_g2_err_fo2.tikz}{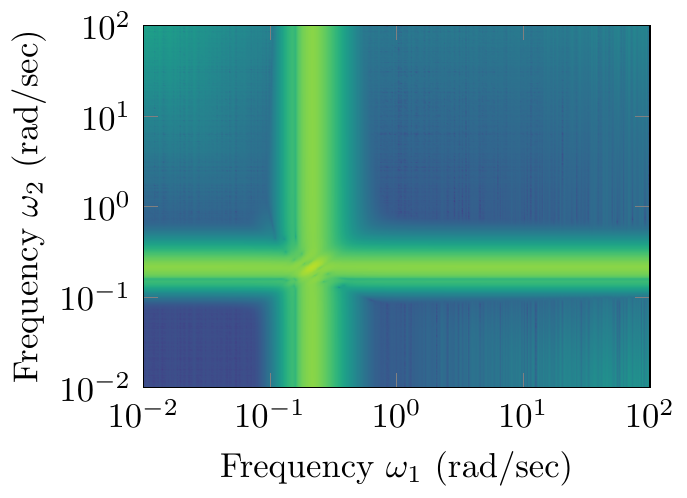}%
    {\tikzset{external/remake next}}{}%
  \begin{tikzpicture}
  \begin{loglogaxis}[
    view   = {0}{90},
    width  = .7\textwidth,
    height = .5\textwidth,
    scale only axis,
    axis on top,
    xmin   = 1e-2,
    xmax   = 1e+2,
    ymin   = 1e-2,
    ymax   = 1e+2,
    xtick  = {1e-2, 1e-1, 1e0, 1e+1, 1e+2},
    ytick  = {1e-2, 1e-1, 1e0, 1e+1, 1e+2},
    xminorticks = false,
    yminorticks = false,
    xlabel = {Frequency $\omega_{1}$ (rad/sec)},
    ylabel = {Frequency $\omega_{2}$ (rad/sec)},
    ylabel style = {yshift = -.3em},
    scaled x ticks = false,
    x tick label style = {/pgf/number format/fixed}]
        
      \addplot graphics[xmin = 1e-2, xmax = 1e+2, ymin = 1e-2, ymax = 1e+2]
        {graphics/data/msd_freq_g2_err_fo2.pdf};
            
  \end{loglogaxis}
\end{tikzpicture}%
  \tikzexternaldisable%

        \subcaption{FOInt(24).}
      \end{center}
    \end{subfigure}
    
    \vspace{.5\baselineskip}

  \tikzexternalenable%
  \tikzsetnextfilename{msd_freq_g2_legend}%
  \filemodCmp{graphics/msd_freq_g2_legend.tikz}{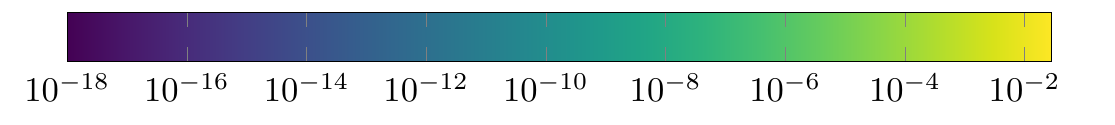}%
    {\tikzset{external/remake next}}{}%
  \begin{tikzpicture}
  \node[draw = none, minimum width = 0cm, inner sep = 0cm](start){};
  \node(leg) at (start.north east) [anchor = north west]{\tikz
  \begin{axis}[%
    hide axis,
    scale only axis,
    width  = 10cm,
    height = .1cm,
    point meta min = -18,
    point meta max = -1.5513,
    colorbar,
    colorbar horizontal,
    colorbar style = {
      xticklabel = $10^{\pgfmathparse{\tick}
        \pgfmathprintnumber\pgfmathresult}$,
      at = {(.5, 0)},
      anchor = north},
    scaled x ticks = false,
    x tick label style = {/pgf/number format/fixed}]
  \end{axis};};
  \node[draw = none, minimum width = .3cm, inner sep = 0cm](end)
    at (leg.north east) [anchor = north west]{};
\end{tikzpicture}%
  \tikzexternaldisable%

    \caption{Relative errors of the second transfer functions for the
      damped mass-spring system.}
    \label{fig:msd_freq_g2}
  \end{center}
\end{figure}
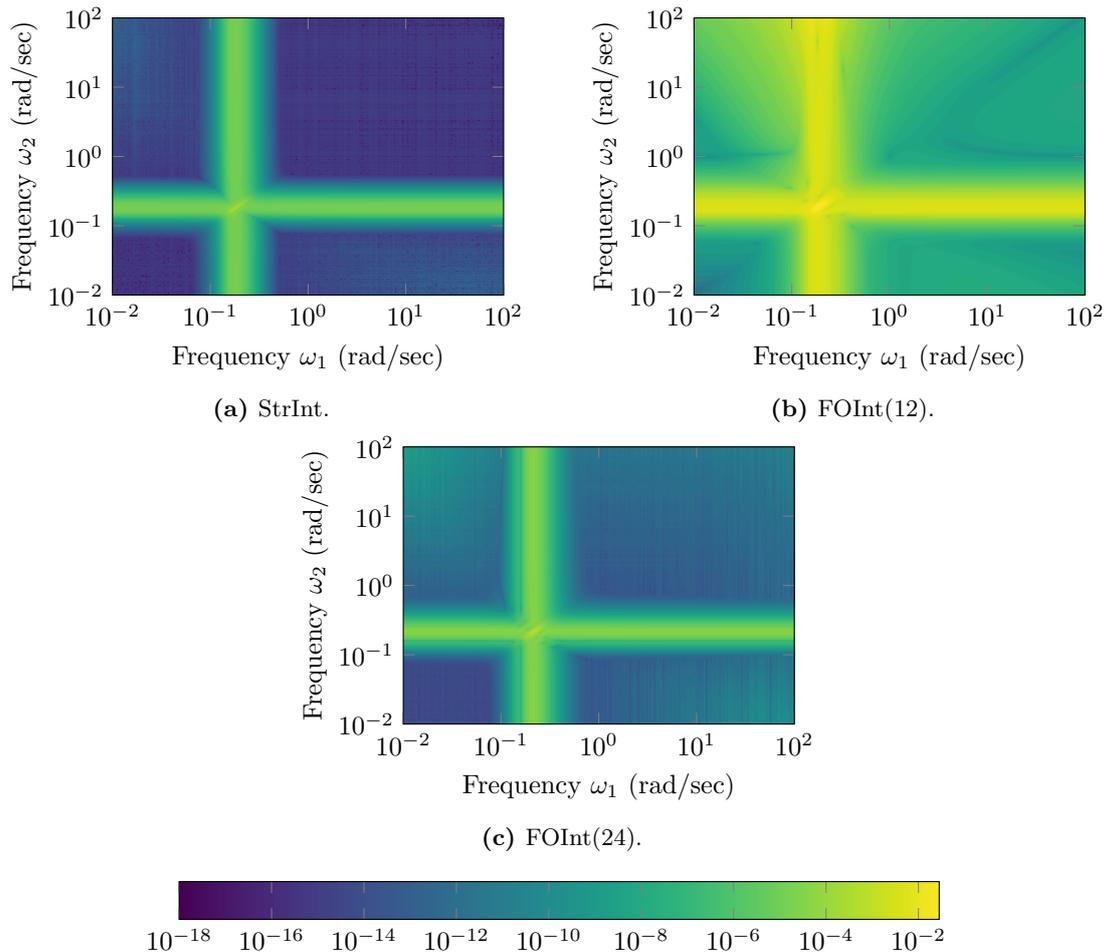


\subsubsection{Time-delayed heated rod}%
\label{sec:siso_td}

This example, taken from~\cite{morGosPBetal19}, models a semi-dis\-cre\-tized 
heated rod  with distributed control and homogeneous Dirichlet boundary 
conditions, which is cooled by a delayed feedback and is described by the PDE
\begin{align*}
  \partial_{t} v(\zeta, t) & = \partial_{\zeta^{2}} v(\zeta, t) - 2\sin(\zeta)
    v(\zeta, t) + 2\sin(\zeta) v(\zeta, t - 1) + u(t),
\end{align*}
with $(\zeta, t) \in (0, \pi) \times (0, t_{f})$ and boundary conditions
$v(0, t) = v(\pi, t) = 0$ for $t \in [0, t_{f}]$.
After a spatial discretization using central finite differences, we obtain a
bilinear time-delay system of the form
\begin{align*}
  \dot{x}(t) & = A x(t) + A_{\rm{d}} x(t-1) + N x(t)u(t) + B u(t),\\
  y(t) & = C x(t),
\end{align*}
with $A, A_{\rm{d}}, N \in \R^{n \times n}$, $B, C^{\trans} \in \R^{n}$, and 
where we have chosen $n = 5\,000$ for our experiments.

To compare with our structure preserving approximation 
(StrInt), in this example, we use the approach from~\cite{morGosPBetal19}
to construct an unstructured bilinear system~\cref{eqn:bsys} without time-delay
using the bilinear Loewner framework, denoted by BiLoewner. 
For the structured interpolation, we have used the interpolation points $\pm
\texttt{logspace(-4, 4, 2)} \iu$ for the first transfer function and
$\pm \texttt{logspace(-2, 2, 2)} \iu$  for the
second transfer function with the two-sided projection approach from 
\Cref{thm:sisovw}.
The resulting reduced-order bilinear time-delay system has order $r = 8$.
For the bilinear Loewner method, we have chosen the interpolation points
$\pm \texttt{logspace(-4, 4, 80)} \iu$ and used the rank truncation idea to
obtain a classical (unstructured) bilinear system, also of order $8$.

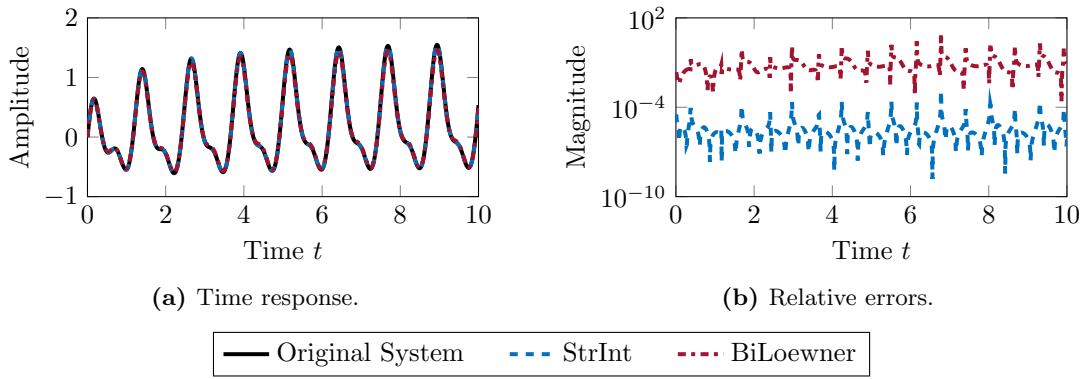
\begin{figure}[tb]
  \begin{center}
    \begin{subfigure}[t]{.49\textwidth}
      \begin{center}
  \tikzexternalenable%
  \tikzsetnextfilename{time_delay_time_sim}%
  \filemodCmp{graphics/time_delay_time_sim.tikz}{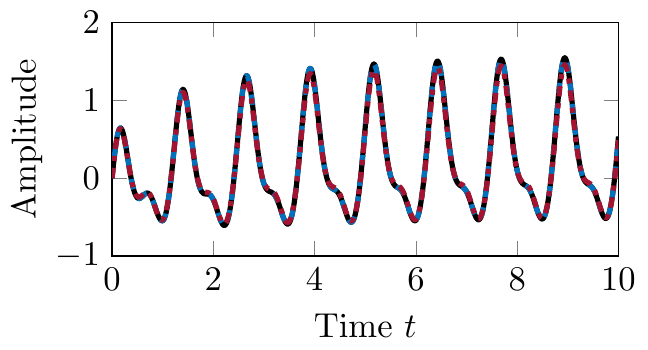}%
    {\tikzset{external/remake next}}{}%
  \begin{tikzpicture}[
  every axis/.append style = {
    scaled x ticks = false,
    x tick label style = {/pgf/number format/fixed},
    cycle list name = plotlist}
  ]
  \pgfplotstableread{graphics/data/time_delay_time_sim.dat}\tableSIM
  
  \begin{axis}[%
    width  = .7\textwidth,
    height = .1\textheight,
    scale only axis,
    xmin = 0,
    xmax = 10,
    ymin = -1,
    ymax = 2,
    xminorticks = false,
    yminorticks = false,
    xlabel = {Time $t$},
    ylabel = {Amplitude},
    ylabel style = {yshift = -.3em}]
        
    \addplot table[x index=0, y index=1] {\tableSIM};
    \addplot table[x index=0, y index=2] {\tableSIM};
    \addplot table[x index=0, y index=3] {\tableSIM};
  \end{axis}
\end{tikzpicture}%
  \tikzexternaldisable%

        \subcaption{Time response.}
      \end{center}
    \end{subfigure}
    \hfill
    \begin{subfigure}[t]{.49\textwidth}
      \begin{center}
  \tikzexternalenable%
  \tikzsetnextfilename{time_delay_time_err}%
  \filemodCmp{graphics/time_delay_time_err.tikz}{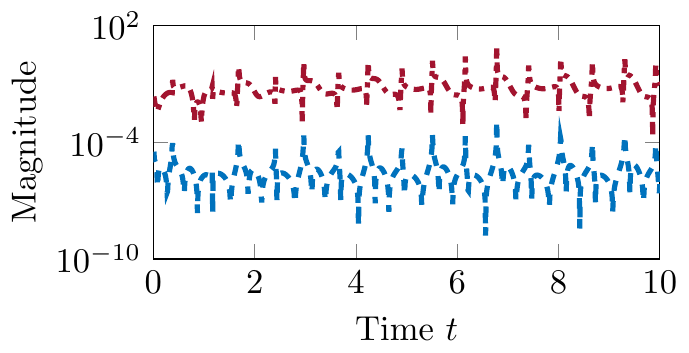}%
    {\tikzset{external/remake next}}{}%
  \begin{tikzpicture}[
  every axis/.append style = {
    scaled x ticks = false,
    x tick label style = {/pgf/number format/fixed},
    cycle list name = plotlist}
  ]
  \pgfplotstableread{graphics/data/time_delay_time_err.dat}\tableRELERR
  
  \begin{semilogyaxis}[%
    width  = .7\textwidth,
    height = .1\textheight,
    scale only axis,
    xmin = 0,
    xmax = 10,
    ymin = 1e-10,
    ymax = 1e+2,
    xminorticks = false,
    yminorticks = false,
    xlabel = {Time $t$},
    ylabel = {Magnitude},
    ylabel style = {yshift = -.3em}]
        
    \pgfplotsset{cycle list shift=1}
    \addplot table[x index=0, y index=1] {\tableRELERR};
    \addplot table[x index=0, y index=2] {\tableRELERR};
  \end{semilogyaxis}
\end{tikzpicture}%
  \tikzexternaldisable%

        \subcaption{Relative errors.}
      \end{center}
    \end{subfigure}
    \vspace{.5\baselineskip}

  \tikzexternalenable%
  \tikzsetnextfilename{time_delay_time_legend}%
  \filemodCmp{graphics/time_delay_time_legend.tikz}{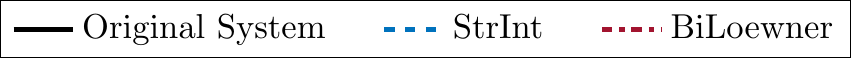}%
    {\tikzset{external/remake next}}{}%
  \begin{tikzpicture}[
  every axis/.append style = {
    scaled x ticks = false,
    x tick label style = {/pgf/number format/fixed},
    cycle list name = plotlist}
  ]
  
  \begin{axis}[%
    hide axis,
    scale only axis,
    width = 1mm,
    legend columns = 3, 
    legend style = {
      at     = {(0,0)},
      anchor = center,
      /tikz/every even column/.append style = {column sep = 0.5cm}}]
    \pgfplotsinvokeforeach{1,...,3}{\addplot coordinates {(0,0)};}
        
    \addlegendentry{Original System};
    \addlegendentry{StrInt};
    \addlegendentry{BiLoewner};
  \end{axis}
\end{tikzpicture}%
  \tikzexternaldisable%

    \caption{Time simulation results for the time-delay system.}
    \label{fig:time_delay_time}
  \end{center}
\end{figure}

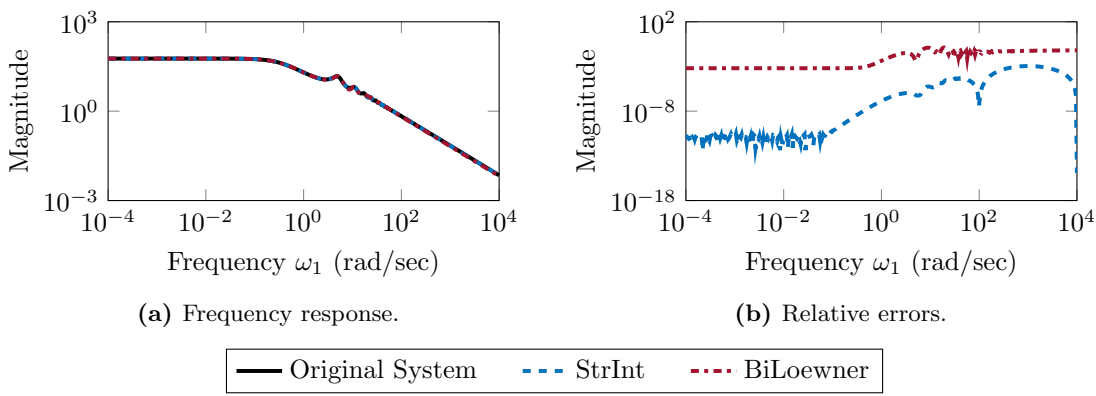
\begin{figure}[tb]
  \begin{center}
    \begin{subfigure}[t]{.49\textwidth}
      \begin{center}
  \tikzexternalenable%
  \tikzsetnextfilename{time_delay_freq_g1_tf}%
  \filemodCmp{graphics/time_delay_freq_g1_tf.tikz}{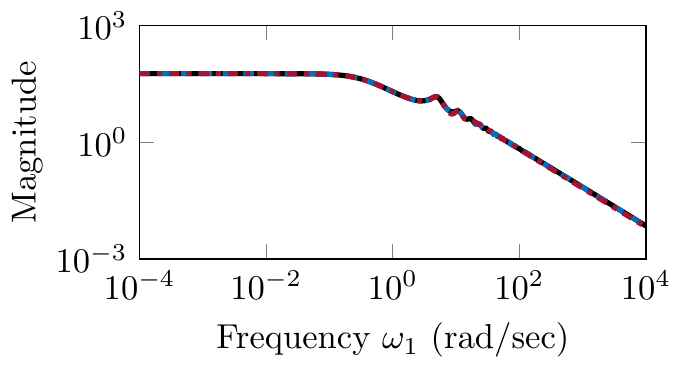}%
    {\tikzset{external/remake next}}{}%
  \begin{tikzpicture}[
  every axis/.append style = {
    scaled x ticks = false,
    x tick label style = {/pgf/number format/fixed},
    cycle list name = plotlist}
  ]
  \pgfplotstableread{graphics/data/time_delay_freq_g1_tf.dat}\tableTF
  
  \begin{loglogaxis}[%
    width  = .7\textwidth,
    height = .1\textheight,
    scale only axis,
    xmin = 1e-4,
    xmax = 1e+4,
    xtick = {1e-4, 1e-2, 1e+0, 1e+2, 1e+4},
    ymin = 1e-3,
    ymax = 1e+3,
    xminorticks = false,
    yminorticks = false,
    xlabel = {Frequency $\omega_{1}$ (rad/sec)},
    ylabel = {Magnitude},
    ylabel style = {yshift = -.3em}]
        
    \addplot table[x index=0, y index=1] {\tableTF};
    \addplot table[x index=0, y index=2] {\tableTF};
    \addplot table[x index=0, y index=3] {\tableTF};
  \end{loglogaxis}
\end{tikzpicture}%
  \tikzexternaldisable%

        \subcaption{Frequency response.}
      \end{center}
    \end{subfigure}
    \hfill
    \begin{subfigure}[t]{.49\textwidth}
      \begin{center}
  \tikzexternalenable%
  \tikzsetnextfilename{time_delay_freq_g1_err}%
  \filemodCmp{graphics/time_delay_freq_g1_err.tikz}{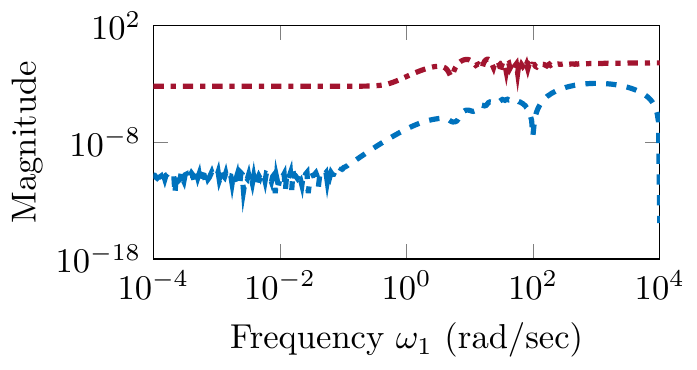}%
    {\tikzset{external/remake next}}{}%
  \begin{tikzpicture}[
  every axis/.append style = {
    scaled x ticks = false,
    x tick label style = {/pgf/number format/fixed},
    cycle list name = plotlist}
  ]
  \pgfplotstableread{graphics/data/time_delay_freq_g1_err.dat}\tableRELERR
  
  \begin{loglogaxis}[%
    width  = .7\textwidth,
    height = .1\textheight,
    scale only axis,
    xmin = 1e-4,
    xmax = 1e+4,
    xtick = {1e-4, 1e-2, 1e+0, 1e+2, 1e+4},
    ymin = 1e-18,
    ymax = 1e+2,
    xminorticks = false,
    yminorticks = false,
    xlabel = {Frequency $\omega_{1}$ (rad/sec)},
    ylabel = {Magnitude},
    ylabel style = {yshift = -.3em}]
    
    \pgfplotsset{cycle list shift=1}
    \addplot table[x index=0, y index=1] {\tableRELERR};
    \addplot table[x index=0, y index=2] {\tableRELERR};
  \end{loglogaxis}
\end{tikzpicture}%
  \tikzexternaldisable%

        \subcaption{Relative errors.}
      \end{center}
    \end{subfigure}
    \vspace{.5\baselineskip}

  \tikzexternalenable%
  \tikzsetnextfilename{time_delay_freq_g1_legend}%
  \filemodCmp{graphics/time_delay_freq_g1_legend.tikz}{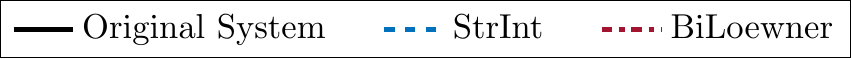}%
    {\tikzset{external/remake next}}{}%
  \begin{tikzpicture}[
  every axis/.append style = {
    scaled x ticks = false,
    x tick label style = {/pgf/number format/fixed},
    cycle list name = plotlist}
  ]
  
  \begin{axis}[%
    hide axis,
    scale only axis,
    width = 1mm,
    legend columns = 3, 
    legend style = {
      at     = {(0,0)},
      anchor = center,
      /tikz/every even column/.append style = {column sep = 0.5cm}}]
    \pgfplotsinvokeforeach{1,...,3}{\addplot coordinates {(0,0)};}
        
    \addlegendentry{Original System};
    \addlegendentry{StrInt};
    \addlegendentry{BiLoewner};
  \end{axis}
\end{tikzpicture}%
  \tikzexternaldisable%

    \caption{Frequency domain results of the first transfer functions for the
      time-delay system.}
    \label{fig:time_delay_freq_g1}
  \end{center}
\end{figure}

With the input signal
\begin{align*}
  u(t) & = \frac{\cos(10 t)}{20} + \frac{\cos(5 t)}{20},
\end{align*}
\Cref{fig:time_delay_time} shows that~\textbf{(a)} the output trajectories 
of the original system, the structure-preserving interpolation and the bilinear 
system without time-delay are indistinguishable in the eye ball 
norm~\textbf{(b)} but the relative error reveals that StrInt is several orders 
of magnitude better than BiLoewner while having the same state-space dimension.
The same behavior can be observed in the frequency domain for the first and
second transfer functions as shown in 
\Cref{fig:time_delay_freq_g1,fig:time_delay_freq_g2}, i.e., by preserving the 
special structure of the original system we obtain a significantly better 
approximation of the same size.

\begin{figure}[tb]
  \begin{center}
    \begin{subfigure}[t]{.49\textwidth}
      \begin{center}
  \tikzexternalenable%
  \tikzsetnextfilename{time_delay_freq_g2_err_rom}%
  \filemodCmp{graphics/time_delay_freq_g2_err_rom.tikz}{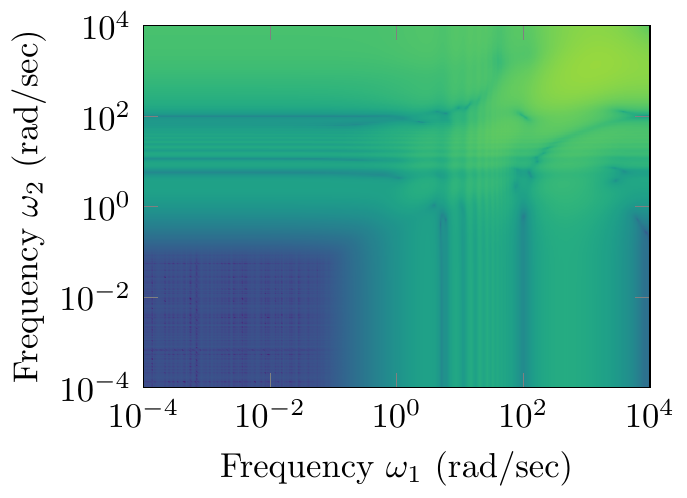}%
    {\tikzset{external/remake next}}{}%
  \begin{tikzpicture}
  \begin{loglogaxis}[
    view   = {0}{90},
    width  = .7\textwidth,
    height = .5\textwidth,
    scale only axis,
    axis on top,
    xmin   = 1e-4,
    xmax   = 1e+4,
    ymin   = 1e-4,
    ymax   = 1e+4,
    xtick  = {1e-4, 1e-2, 1e0, 1e+2, 1e+4},
    ytick  = {1e-4, 1e-2, 1e0, 1e+2, 1e+4},
    xlabel = {Frequency $\omega_{1}$ (rad/sec)},
    ylabel = {Frequency $\omega_{2}$ (rad/sec)},
    ylabel style = {yshift = -.3em},
    scaled x ticks = false,
    x tick label style = {/pgf/number format/fixed}]
        
      \addplot graphics[xmin = 1e-4, xmax = 1e+4, ymin = 1e-4, ymax = 1e+4]
        {graphics/data/time_delay_freq_g2_err_rom.pdf};
            
  \end{loglogaxis}
\end{tikzpicture}%
  \tikzexternaldisable%

        \subcaption{StrInt.}
      \end{center}
    \end{subfigure}
    \hfill
    \begin{subfigure}[t]{.49\textwidth}
      \begin{center}
  \tikzexternalenable%
  \tikzsetnextfilename{time_delay_freq_g2_err_lwn}%
  \filemodCmp{graphics/time_delay_freq_g2_err_lwn.tikz}{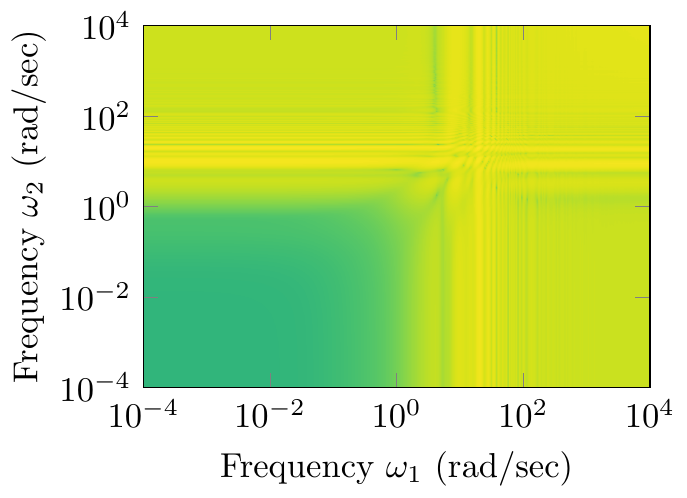}%
    {\tikzset{external/remake next}}{}%
  \begin{tikzpicture}
  \begin{loglogaxis}[
    view   = {0}{90},
    width  = .7\textwidth,
    height = .5\textwidth,
    scale only axis,
    axis on top,
    xmin   = 1e-4,
    xmax   = 1e+4,
    ymin   = 1e-4,
    ymax   = 1e+4,
    xtick  = {1e-4, 1e-2, 1e0, 1e+2, 1e+4},
    ytick  = {1e-4, 1e-2, 1e0, 1e+2, 1e+4},
    xlabel = {Frequency $\omega_{1}$ (rad/sec)},
    ylabel = {Frequency $\omega_{2}$ (rad/sec)},
    ylabel style = {yshift = -.3em},
    scaled x ticks = false,
    x tick label style = {/pgf/number format/fixed}]
        
      \addplot graphics[xmin = 1e-4, xmax = 1e+4, ymin = 1e-4, ymax = 1e+4]
        {graphics/data/time_delay_freq_g2_err_lwn.pdf};
            
  \end{loglogaxis}
\end{tikzpicture}%
  \tikzexternaldisable%

        \subcaption{BiLoewner.}
      \end{center}
    \end{subfigure}
    \vspace{.5\baselineskip}

  \tikzexternalenable%
  \tikzsetnextfilename{time_delay_freq_g2_legend}%
  \filemodCmp{graphics/time_delay_freq_g2_legend.tikz}{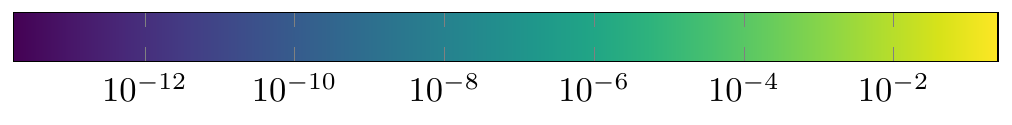}%
    {\tikzset{external/remake next}}{}%
  \begin{tikzpicture}
  \node[draw = none, minimum width = 0cm, inner sep = 0cm](start){};
  \node(leg) at (start.north east) [anchor = north west]{\tikz
  \begin{axis}[%
    hide axis,
    scale only axis,
    width  = 10cm,
    height = .1cm,
    point meta min = -13.7539,
    point meta max = -0.6086,
    colorbar,
    colorbar horizontal,
    colorbar style = {
      xticklabel = $10^{\pgfmathparse{\tick}
        \pgfmathprintnumber\pgfmathresult}$,
      at = {(.5, 0)},
      anchor = north},
    scaled x ticks = false,
    x tick label style = {/pgf/number format/fixed}]
  \end{axis};};
  \node[draw = none, minimum width = 0cm, inner sep = 0cm](end)
    at (leg.north east) [anchor = north west]{};
\end{tikzpicture}%
  \tikzexternaldisable%

    \caption{Relative errors of the second transfer functions for the
      time-delay system.}
    \label{fig:time_delay_freq_g2}
  \end{center}
\end{figure}
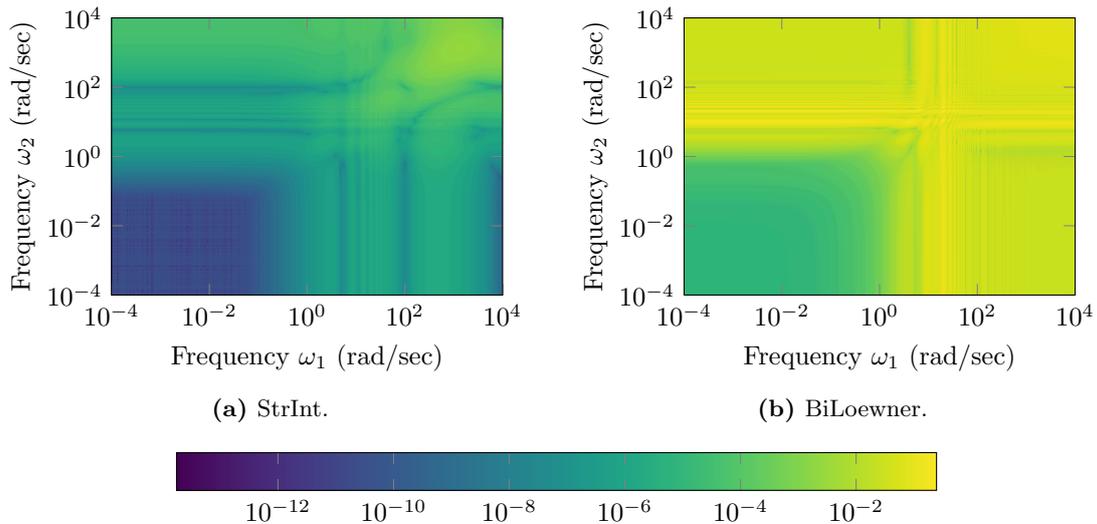


\section{Interpolation of multi-input multi-output systems}%
\label{sec:mimo}

In this section, we will generalize the results from  SISO structured bilinear 
systems to MIMO ones as in~\cref{eqn:mimotf} and give a numerical example to 
illustrate the theory.


\subsection{Matrix interpolation}

In principle, all the results from \Cref{sec:siso} can directly be
extended to the MIMO system case~\cref{eqn:mimotf}. However,
one needs to realize that in this case, the quantities to be interpolated, i.e., 
the subsystem transfer functions, are matrix-valued.
The main difference from the SISO case lies in the collection of the bilinear
matrices into $\cN(s) = \begin{bmatrix} \cN_{1}(s) & \ldots & \cN_{m} (s)
\end{bmatrix}$ and the corresponding Kronecker products that produce the
different combinations of the linear and bilinear parts in the $k$-th order
transfer functions, e.g., in~\cref{eqn:nokron}.
Additionally, we will use the following notation
\begin{align*}
  \tcN(s) := \begin{bmatrix} \cN_{1}(s) \\ \vdots \\ \cN_{m}(s) \end{bmatrix}
\end{align*}
as alternative way of concatenating the bilinear terms.
In this paper, we will only focus on matrix interpolation, i.e., we will 
interpolate the full matrix-valued structured subsystem transfer functions.
There is a concept of tangential interpolation~\cite{morGalVV04, morAntBG20} to 
handle matrix-valued functions in which interpolation is enforced only in 
selected directions.
We will consider that framework in a separate work since the definition of 
tangential interpolation is not unified yet for bilinear 
systems~\cite{morBenBD11, morRodGB18}, let alone the structured ones we consider 
here.

The following theorem extends the results from
\Cref{thm:sisov,thm:sisow,thm:sisovw} to MIMO structured bilinear systems.

\begin{theorem}[Matrix interpolation]%
  \label{thm:mimovw}
  Let $G$ be a bilinear system, as described by~\cref{eqn:mimotf}, and $\hG$
  the reduced-order bilinear system, constructed by~\cref{eqn:proj}.
  Given sets of interpolation points $\sigma_{1}, \ldots, \sigma_{k} \in \C$ and
  $\varsigma_{1}, \ldots, \varsigma_{\theta} \in \C$, for which the matrix 
  functions $\cC(s)$, $\cK(s)^{-1}$, $\cN(s)$, $\cB(s)$ are defined 
  and $\hcK(s)$ is full-rank, the following statements hold:
  \begin{enumerate}[label=(\alph*)] 
    \item If $V$ is constructed as
      \begin{align*}
        V_{1} & = \cK(\sigma_{1})^{-1}\cB(\sigma_{1}),\\
        V_{j} & = \cK(\sigma_{j})^{-1}\cN(\sigma_{j-1})
          (I_{m} \otimes V_{j-1}), & 2 \leq j \leq k,\\
        \mspan(V) & \supseteq \mspan\left([V_{1}, \ldots,
        V_{k}]\right),
      \end{align*}
      then the following interpolation conditions hold true:
      \begin{align*}
        \begin{aligned}
          G_{1}(\sigma_{1}) & =  \hG_{1}(\sigma_{1}),\\
          G_{2}(\sigma_{1},\sigma_{2}) & = \hG_{2}(\sigma_{1}, \sigma_{2}),\\
          & \,\,\,\vdots\\
          G_{k}(\sigma_{1}, \ldots, \sigma_{k}) &
            = \hG_{k}(\sigma_{1}, \ldots, \sigma_{k}).
        \end{aligned}
      \end{align*}
    \item If $W$ is constructed as
      \begin{align*}
        W_{1} & = \cK(\varsigma_{\theta})^{-\herm}\cC
          (\varsigma_{\theta})^{\herm},\\
        W_{i} & = \cK(\varsigma_{\theta-i+1})^{-\herm}
          \tcN(\varsigma_{k-i+1})^{\herm} (I_{m} \otimes W_{i-1}),
          & 2 \leq i \leq \theta,\\
        \mathrm{span}(W) & \supseteq \mathrm{span}\left([W_{1}, \ldots,
          W_{\theta}] \right),
      \end{align*}
      then the following interpolation conditions hold true:
      \begin{align*}
        \begin{aligned}
          G_{1}(\varsigma_{\theta}) & =  \widehat{G}_{1}(\varsigma_{\theta}),\\
          G_{2}(\varsigma_{\theta-1},\sigma_{\theta}) &
            = \widehat{G}_{2}(\varsigma_{\theta-1}, \sigma_{\theta}),\\
          & \,\,\,\vdots\\
          G_{\theta}(\varsigma_{1}, \ldots, \varsigma_{\theta}) &
            = \widehat{G}_{\theta}(\varsigma_{1}, \ldots, \varsigma_{\theta}).
        \end{aligned}
      \end{align*}
    \item Let $V$ be constructed as in part (a) and $W$ as in (b), then,
      additionally to the results in (a) and (b), the interpolation conditions
      \begin{align*}
        G_{q + \eta}(\sigma_{1}, \ldots, \sigma_{q},
          \varsigma_{\theta-\eta+1}, \ldots, \varsigma_{\theta}) & =
          \hG_{q + \eta}(\sigma_{1}, \ldots, \sigma_{q},
          \varsigma_{\theta-\eta+1}, \ldots, \varsigma_{\theta}),
      \end{align*}
      hold for $1 \leq q \leq k$ and  $1 \leq \eta \leq \theta$.
  \end{enumerate}
\end{theorem}
\begin{proof}
  Starting with part (a), we remember that the transfer functions can be
  rewritten by multiplying out the Kronecker products as
  \begin{align*}
    G_{k}(\sigma_{1}, \ldots, \sigma_{k}) & = [ \cC(\sigma_{k})
      \cK(\sigma_{k})^{-1}\cN_{1}(\sigma_{k-1}) \cdots \cN_{1}(\sigma_{1})
      \cK(\sigma_{1})^{-1}\cB(\sigma_{1})^{-1}, \\
    & ~~~~~ \cC(\sigma_{k})\cK(\sigma_{k})^{-1}\cN_{1}(\sigma_{k-1}) \cdots
      \cN_{2}(\sigma_{1})\cK(\sigma_{1})^{-1}\cB(\sigma_{1})^{-1},\\
    & ~~~~~ \cdots\\
    & ~~~~~ \cC(\sigma_{k})\cK(\sigma_{k})^{-1}\cN_{m}(\sigma_{k-1}) \cdots
      \cN_{m}(\sigma_{1})\cK(\sigma_{1})^{-1}\cB(\sigma_{1})^{-1}].
  \end{align*}
  From the construction of $V$, it follows that applying \Cref{thm:sisov}
  for the transfer functions in each single entry gives the result.
  Part (b) directly follows from part (a) by replacing the matrix functions
  by their Hermitian conjugate versions except for $\cN(s) = 
  \begin{bmatrix} \cN_{1} & \ldots & \cN_{m} \end{bmatrix}$, where the single 
  entries have to be transposed conjugated.
  Therefore, the differently stacked $\tcN(s)$ is used here to give
  $\tcN(s)^{\herm} = \begin{bmatrix} \cN_{1}(s)^{\herm} & \ldots &
  \cN_{m}(s)^{\herm} \end{bmatrix}$. Finally,
  Part (c) follows directly from part (a), (b) and \Cref{thm:sisovw}
  for the single transfer function entries.
\end{proof}

For Hermite interpolation as in
\Cref{thm:sisovhermite,thm:sisowhermite,thm:sisovwderivative}, a similar 
extension to the MIMO case follows.

\begin{theorem}[Hermite matrix interpolation]%
  \label{thm:mimovwderivative}
  Let $G$ be a bilinear system, described by~\cref{eqn:mimotf}, and $\hG$ the
  reduced-order bilinear system, constructed by~\cref{eqn:proj}.
  Given sets of interpolation points $\sigma_{1}, \ldots, \sigma_{k} \in \C$ and
  $\varsigma_{1}, \ldots, \varsigma_{\theta} \in \C$, for which the matrix 
  functions $\cC(s)$, $\cK(s)^{-1}$, $\cN(s)$, $\cB(s)$ are analytic and 
  $\hcK(s)$ is full-rank, the following statements hold:
  \begin{enumerate}[label=(\alph*)]
    \item If $V$ is constructed as
      \begin{align*}
        V_{1, j_{1}} & = \partial_{s^{j_{1}}} (\cK^{-1} \cB) (\sigma_{1}),
          & j_{1} & = 0, \ldots, \ell_{1},\\
        V_{2, j_{2}} & = \partial_{s^{j_{2}}} \cK^{-1} (\sigma_{2})
          \partial_{s^{\ell_{1}}} (\cN (I_{m} \otimes \cK^{-1} \cB))
          (\sigma_{1}),
          & j_{2} & = 0,\ldots,\ell_{2},\\
        & \,\,\,\vdots\\
        V_{k, j_{k}} & = \partial_{s^{j_{k}}} \cK^{-1} (\sigma_{k})
          \left( \prod\limits_{j = 1}^{k-2} \partial_{s^{\ell_{k -j}}}
          \big( (I_{m^{j-1}} \otimes \cN) (I_{m^{j}} \otimes \cK) \big)
          (\sigma_{k-j}) \right)\\
        & \quad{}\times{} \partial_{s^{\ell_{1}}}
          ((I_{m^{k-2}} \otimes \cN)(I_{m^{k-1}} \otimes \cK)
          (I_{m^{k-1}} \otimes \cB)) (\sigma_{1}),
          & j_{k} & = 0, \ldots, \ell_{k},\\
        \mspan(V) & \supseteq \mspan([V_{1,0}, \ldots, V_{k, \ell_{k}}]),
      \end{align*}
      then the following interpolation conditions hold true:
      \begin{align*}
        \partial_{s_{1}^{j_{1}}} G_{1} (\sigma_{1})
           & = \partial_{s_{1}^{j_{1}}} \hG_{1} (\sigma_{1}),
           & j_{1} = 0, \ldots, \ell_{1},\\
         & \,\,\,\vdots\\
         \partial_{s_{1}^{\ell_{1}} \cdots s_{k-1}^{\ell_{k-1}} s_{k}^{j_{k}}}
           G_{k} (\sigma_{1}, \ldots, \sigma_{k})
           & = \partial_{s_{1}^{\ell_{1}} \cdots s_{k-1}^{\ell_{k-1}}
           s_{k}^{j_{k}}} \hG_{k} (\sigma_{1}, \ldots, \sigma_{k}),
           & j_{k} = 0,\ldots,\ell_{k}.
      \end{align*}
    \item If $W$ is constructed as
      \begin{align*}
        W_{1, i_{\theta}} & = \partial_{s^{i_{\theta}}} (\cK^{-\herm}
          \cC^{\herm}) (\varsigma_{\theta}),
          & i_{\theta} & = 0,\ldots,\nu_{\theta},\\
        W_{2, i_{\theta-1}} & = \partial_{s^{i_{\theta-1}}} (\cK^{-\herm}
          \tcN^{\herm}) (\varsigma_{\theta-1})
          \left( I_{m} \otimes \partial_{s^{\nu_{\theta}}} (\cK^{-\herm}
          \cC^{\herm}) (\varsigma_{\theta}) \right),
        & i_{\theta-1} & = 0,\ldots,\nu_{\theta-1},\\
        & \,\,\,\vdots\\
        W_{\theta, i_{1}} & = \partial_{s^{i_{1}}} (\cK^{-\herm} \tcN^{\herm})
          (\varsigma_{1}) \left( \prod\limits_{i = 2}^{\theta - 1}
          \partial_{s^{\nu_{i}}} (I_{m^{i-1}} \otimes \cK^{-\herm} \tcN^{\herm})
          (\varsigma_{i}) \right)\\
        & \quad{}\times{} \left( I_{m^{\theta-1}} \otimes
          \partial_{s^{\nu_{\theta}}} (\cK^{-\herm} \cC^{\herm})
          (\varsigma_{\theta}) \right),
          & i_{1} & = 0, \ldots, \nu_{1},\\
        \mathrm{span}(W) & \supseteq \mspan([W_{1,0}, \ldots, 
          W_{\theta, \nu_{\theta}}]),
      \end{align*}
      then the following interpolation conditions hold true:
      \begin{align*}
        \partial_{s_{1}^{i_{\theta}}} G_{1} (\varsigma_{\theta})
          & = \partial_{s_{1}^{i_{\theta}}} \hG_{1} (\varsigma_{\theta}),
          & i_{\theta} & = 0, \ldots, \nu_{\theta},\\
        & \,\,\,\vdots\\
        \partial_{s_{1}^{i_{1}} s_{2}^{\nu_{2}} \cdots
          s_{\theta}^{\nu_{\theta}}} G_{\theta}
          (\varsigma_{1}, \ldots, \varsigma_{\theta})
          & = \partial_{s_{1}^{i_{1}} s_{2}^{\nu_{2}} \cdots
          s_{\theta}^{\nu_{\theta}}} \hG_{\theta}
          (\varsigma_{1}, \ldots, \varsigma_{\theta}),
          & i_{1} & = 0,\ldots,\nu_{1}.
      \end{align*}
    \item Let $V$ be constructed as in part (a) and $W$ as in part (b), then,
      additionally to the results in (a) and (b), the conditions
      \begin{align*}
        & \partial_{s_{1}^{\ell_{1}} \cdots s_{q-1}^{\ell_{q-1}}
          s_{q}^{j_{q}} s_{q+1}^{i_{\theta - \eta + 1}}
          s_{q+2}^{\nu_{\theta - \eta + 2}} \cdots s_{q + \eta}^{\nu_{\theta}}}
          G_{q + \eta} (\sigma_{1}, \ldots, \sigma_{q},
          \varsigma_{\theta - \eta + 1}, \ldots, \varsigma_{\theta})\\
        & = \partial_{s_{1}^{\ell_{1}} \cdots s_{q-1}^{\ell_{q-1}}
          s_{q}^{j_{q}} s_{q+1}^{i_{\theta - \eta + 1}}
          s_{q+2}^{\nu_{\theta - \eta + 2}} \cdots s_{q + \eta}^{\nu_{\theta}}}
          \hG_{q + \eta} (\sigma_{1}, \ldots, \sigma_{q},
          \varsigma_{\theta - \eta + 1}, \ldots, \varsigma_{\theta}),
      \end{align*}
      hold for $j_{q} = 0, \ldots, \ell_{q}$; $i_{\theta - \eta + 1} = 0, \ldots,
      \nu_{\theta - \eta + 1}$; $1 \leq q \leq k$ and $1 \leq \eta \leq \theta$.
  \end{enumerate}
\end{theorem}
\begin{proof}
  The results follow directly from
  \Cref{thm:sisovhermite,thm:sisowhermite,thm:sisovwderivative} with the same 
  argumentation as in \Cref{thm:mimovw}.
\end{proof}

For completeness, also the implicit interpolation results are stated in
the following corollary without additional proofs.

\begin{corollary}[Two-sided matrix interpolation with identical point sets]
  Let $G$ be a bilinear system, described by~\cref{eqn:mimotf}, and $\hG$ the
  reduced-order bilinear system, constructed by~\cref{eqn:proj}.
  Given a set of interpolation points $\sigma_{1}, \ldots, \sigma_{k} \in \C$, 
  for which the matrix functions $\cC(s)$, $\cK(s)^{-1}$, $\cN(s)$, $\cB(s)$ are
  analytic and $\hcK(s)$ is full-rank, the following statements hold:
  \begin{enumerate}[label=(\alph*)]
    \item Let $V$ and $W$ be constructed as in \Cref{thm:mimovw} (a)
      and (b) for the interpolation points $\sigma_{1}, \ldots, \sigma_{k}$,
      then additionally it holds
      \begin{align*}
        \nabla G_{k}(\sigma_{1}, \ldots, \sigma_{k})
          & = \nabla \hG_{k}(\sigma_{1}, \ldots, \sigma_{k}).
      \end{align*}
    \item Let $V$ and $W$ be constructed as in
      \Cref{thm:mimovwderivative} (a) and (b) for the interpolation
      points $\sigma_{1}, \ldots, \sigma_{k}$, then additionally it holds
      \begin{align*}
        \nabla \left( \partial_{s_{1}^{\ell_{1}} \cdots s_{k}^{\ell_{k}}}
          G_{k} (\sigma_{1}, \ldots, \sigma_{k}) \right)
          & = \nabla \left( \partial_{s_{1}^{\ell_{1}} \cdots s_{k}^{\ell_{k}}}
          \hG_{k} (\sigma_{1}, \ldots, \sigma_{k}) \right).
      \end{align*}
  \end{enumerate}
\end{corollary}


\subsection{Numerical example}

We  illustrate the matrix interpolation results in a numerical example.
The experiments reported here have been executed on the same machine and with 
the same MATLAB version as in \Cref{sec:sisoexamples}.

We reconsider the damped mass-spring system example from 
\Cref{sec:siso_msd} with the following modifications:
The mass, damping and stiffness matrices from~\cref{eqn:msd} stay unchanged.
The input forces are now applied to the first and last masses, i.e., the input
term becomes $B_{\rm{u}} = \begin{bmatrix} e_{1}, -e_{n} \end{bmatrix}$, and
we observe the displacement of the second and fifth masses, which
gives the output matrix $C_{\rm{p}} = \begin{bmatrix} e_{2}, e_{5} 
\end{bmatrix}^{\trans}$.
Therefore, we have $2$ inputs and outputs. 
We consider the same idea of bilinear springs as before but working in
different directions, i.e., we have
\begin{align*}
  \begin{aligned}
    N_{\rm{p},1} & = -S_{1}KS_{1} && \text{and} & N_{\rm{p},2}& = S_{2}KS_{2},
  \end{aligned}
\end{align*}
where $S_{1}$ is chosen, as before, as diagonal matrix with
\texttt{linspace(0.2, 0, n)}, and $S_{2}$ is chosen to be a diagonal matrix with 
\texttt{linspace(0, 0.2, n)} as entries.
Overall, we have a damped mass-spring system of the form
\begin{align} \label{eqn:msd_mimo}
  \begin{aligned}
    M\ddot{x}(t) + D\dot{x}(t) + Kx(t) & = N_{\rm{p},1}x(t)u_{1}(t) +
      N_{\rm{p},2}x(t)u_{2}(t) + B_{\rm{u}}u(t),\\
    y(t) & = C_{\rm{p}}x(t),
  \end{aligned}
\end{align}
with $n = 1\,000$ masses for our experiments.

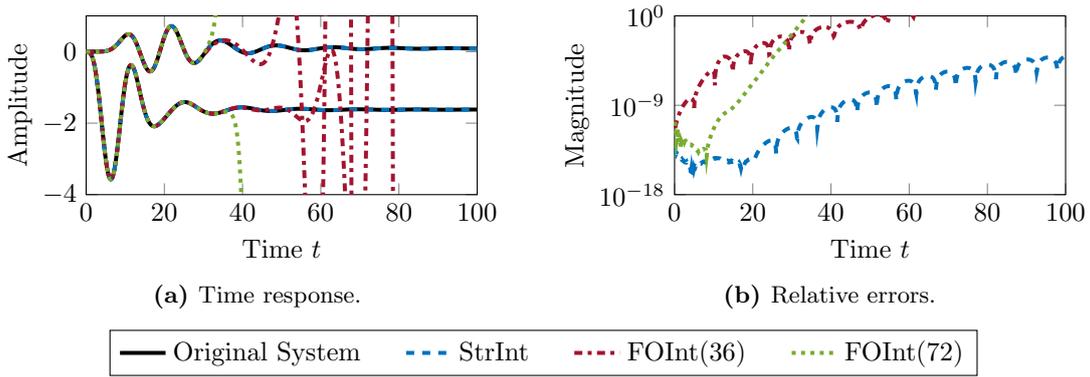
\begin{figure}[tb]
  \begin{center}
    \begin{subfigure}[t]{.49\textwidth}
      \begin{center}
  \tikzexternalenable%
  \tikzsetnextfilename{mimo_msd_time_sim}%
  \filemodCmp{graphics/mimo_msd_time_sim.tikz}{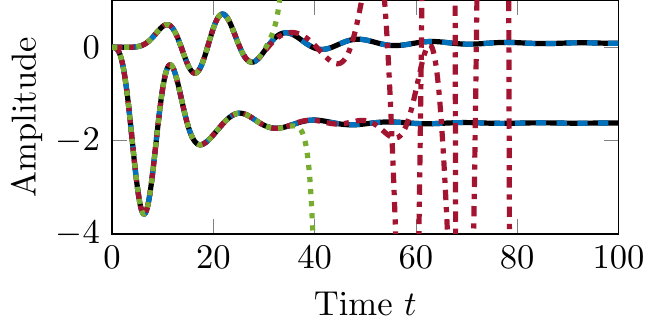}%
    {\tikzset{external/remake next}}{}%
  \begin{tikzpicture}[
  every axis/.append style = {
    scaled x ticks = false,
    x tick label style = {/pgf/number format/fixed},
    cycle list name = plotlist}
  ]
  \pgfplotstableread{graphics/data/mimo_msd_time_sim.dat}\tableSIM
  
  \begin{axis}[%
    width  = .7\textwidth,
    height = .1\textheight,
    scale only axis,
    xmin = 0,
    xmax = 100,
    restrict y to domain = -7:4,
    ymin = -4,
    ymax = 1,
    xminorticks = false,
    yminorticks = false,
    xlabel = {Time $t$},
    ylabel = {Amplitude},
    ylabel style = {yshift = -.3em}]
        
    \addplot table[x index = 0, y index=1] {\tableSIM};
    \addplot table[x index = 0, y index=3] {\tableSIM};
    \addplot table[x index = 0, y index=5] {\tableSIM};
    \addplot table[x index = 0, y index=7] {\tableSIM};
    
    \addplot table[x index=0, y index=2] {\tableSIM};
    \addplot table[x index=0, y index=4] {\tableSIM};
    \addplot table[x index=0, y index=6] {\tableSIM};
    \addplot table[x index=0, y index=8] {\tableSIM};
  \end{axis}
\end{tikzpicture}%
  \tikzexternaldisable%

        \subcaption{Time response.}
        \label{fig:mimo_msd_time_sim}
      \end{center}
    \end{subfigure}
    \hfill
    \begin{subfigure}[t]{.49\textwidth}
      \begin{center}
  \tikzexternalenable%
  \tikzsetnextfilename{mimo_msd_time_err}%
  \filemodCmp{graphics/mimo_msd_time_err.tikz}{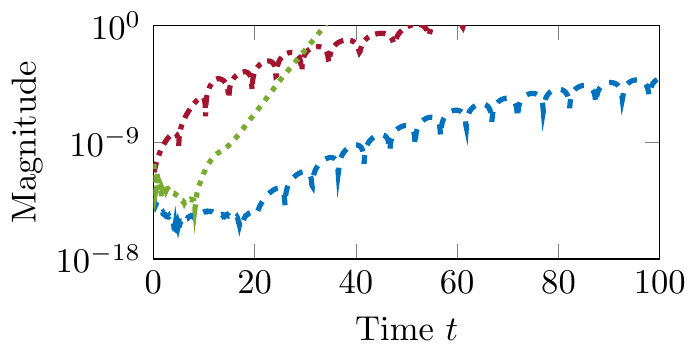}%
    {\tikzset{external/remake next}}{}%
  \begin{tikzpicture}[
  every axis/.append style = {
    scaled x ticks = false,
    x tick label style = {/pgf/number format/fixed},
    cycle list name = plotlist}
  ]
  \pgfplotstableread{graphics/data/mimo_msd_time_err.dat}\tableRELERR
  
  \begin{semilogyaxis}[%
    width  = .7\textwidth,
    height = .1\textheight,
    scale only axis,
    xmin = 0,
    xmax = 100,
    ymin = 1e-18,
    ymax = 1,
    xminorticks = false,
    yminorticks = false,
    xlabel = {Time $t$},
    ylabel = {Magnitude},
    ylabel style = {yshift = -.3em}]
    
    \pgfplotsset{cycle list shift = 1}
    \addplot table[x index = 0, y index = 1] {\tableRELERR};
    \addplot table[x index = 0, y index = 2] {\tableRELERR};
    \addplot table[x index = 0, y index = 3] {\tableRELERR};
  \end{semilogyaxis}
\end{tikzpicture}%
  \tikzexternaldisable%

        \subcaption{Relative errors.}
      \end{center}
    \end{subfigure}
    \vspace{.5\baselineskip}

  \tikzexternalenable%
  \tikzsetnextfilename{mimo_msd_time_legend}%
  \filemodCmp{graphics/mimo_msd_time_legend.tikz}{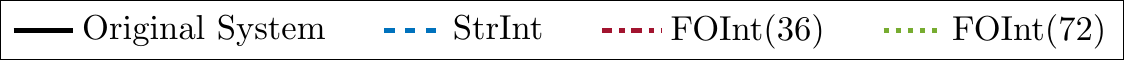}%
    {\tikzset{external/remake next}}{}%
  \begin{tikzpicture}[
  every axis/.append style = {
    scaled x ticks = false,
    x tick label style = {/pgf/number format/fixed},
    cycle list name = plotlist}
  ]
  
  \begin{axis}[%
    hide axis,
    scale only axis,
    width = 1mm,
    legend columns = 4, 
    legend style = {
      at     = {(0,0)},
      anchor = center,
      /tikz/every even column/.append style = {column sep = 0.5cm}}]
    \pgfplotsinvokeforeach{1,...,4}{\addplot coordinates {(0,0)};}
    
    \addlegendentry{Original System};
    \addlegendentry{StrInt};
    \addlegendentry{FOInt(36)};
    \addlegendentry{FOInt(72)};
  \end{axis}
\end{tikzpicture}%
  \tikzexternaldisable%

    \caption{Time simulation results for the MIMO damped mass-spring system.}
    \label{fig:mimo_msd_time}
  \end{center}
\end{figure}

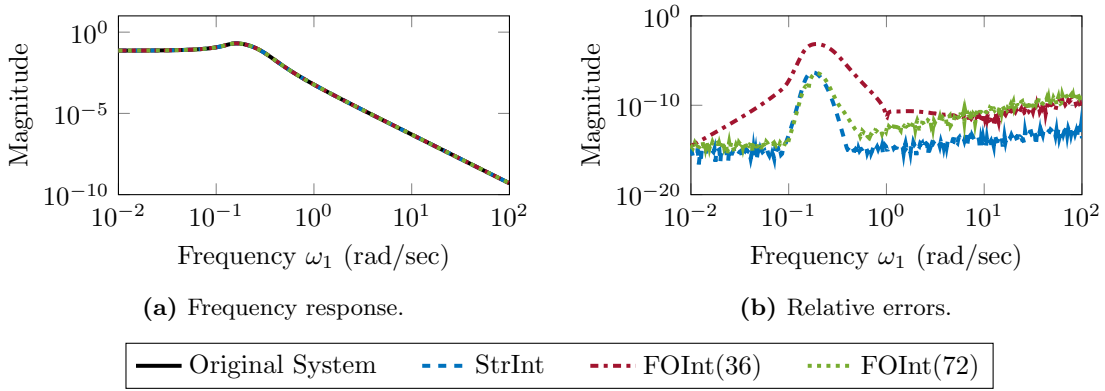
\begin{figure}[tb]
  \begin{center}
    \begin{subfigure}[t]{.49\textwidth}
      \begin{center}
  \tikzexternalenable%
  \tikzsetnextfilename{mimo_msd_freq_g1_tf}%
  \filemodCmp{graphics/mimo_msd_freq_g1_tf.tikz}{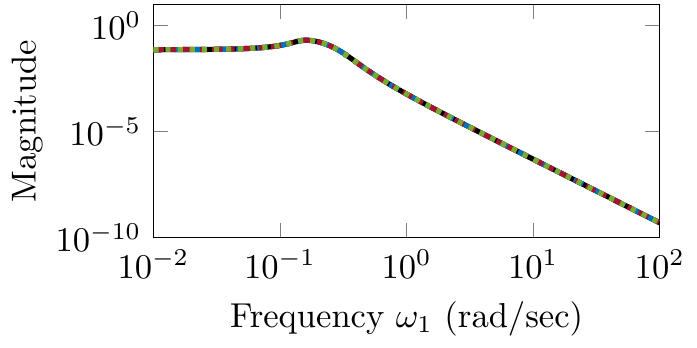}%
    {\tikzset{external/remake next}}{}%
  \begin{tikzpicture}[
  every axis/.append style = {
    scaled x ticks = false,
    x tick label style = {/pgf/number format/fixed},
    cycle list name = plotlist}
  ]
  \pgfplotstableread{graphics/data/mimo_msd_freq_g1_tf.dat}\tableTF

  \begin{loglogaxis}[%
    width  = .7\textwidth,
    height = .1\textheight,
    scale only axis,
    xmin = 1e-2,
    xmax = 1e+2,
    ymin = 1e-10,
    ymax = 1e+1,
    xminorticks = false,
    yminorticks = false,
    xlabel = {Frequency $\omega_{1}$ (rad/sec)},
    ylabel = {Magnitude},
    ylabel style = {yshift = -.3em}]
        
    \addplot table[x index = 0, y index = 1] {\tableTF};
    \addplot table[x index = 0, y index = 2] {\tableTF};
    \addplot table[x index = 0, y index = 3] {\tableTF};
    \addplot table[x index = 0, y index = 4] {\tableTF};
  \end{loglogaxis}
\end{tikzpicture}%
  \tikzexternaldisable%

        \subcaption{Frequency response.}
      \end{center}
    \end{subfigure}
    \hfill
    \begin{subfigure}[t]{.49\textwidth}
      \begin{center}
  \tikzexternalenable%
  \tikzsetnextfilename{mimo_msd_freq_g1_err}%
  \filemodCmp{graphics/mimo_msd_freq_g1_err.tikz}{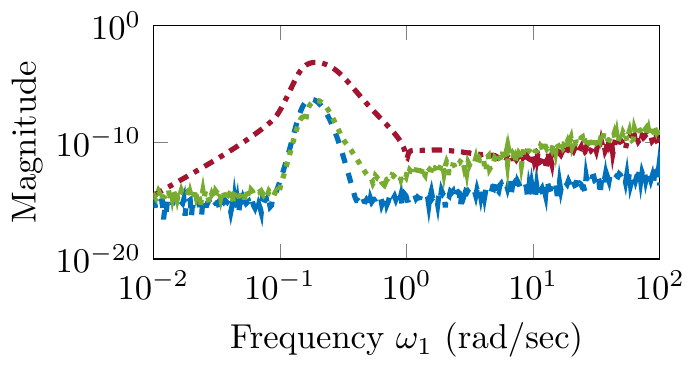}%
    {\tikzset{external/remake next}}{}%
  \begin{tikzpicture}[
  every axis/.append style = {
    scaled x ticks = false,
    x tick label style = {/pgf/number format/fixed},
    cycle list name = plotlist}
  ]
  \pgfplotstableread{graphics/data/mimo_msd_freq_g1_err.dat}\tableRELERR
  
  \begin{loglogaxis}[%
    width  = .7\textwidth,
    height = .1\textheight,
    scale only axis,
    xmin = 1e-2,
    xmax = 1e+2,
    ymin = 1e-20,
    ymax = 1,
    xminorticks = false,
    yminorticks = false,
    xlabel = {Frequency $\omega_{1}$ (rad/sec)},
    ylabel = {Magnitude},
    ylabel style = {yshift = -.3em}]
        
    \pgfplotsset{cycle list shift = 1}
    \addplot table[x index = 0, y index = 1] {\tableRELERR};
    \addplot table[x index = 0, y index = 2] {\tableRELERR};
    \addplot table[x index = 0, y index = 3] {\tableRELERR};
  \end{loglogaxis}
\end{tikzpicture}%
  \tikzexternaldisable%

        \subcaption{Relative errors.}
      \end{center}
    \end{subfigure}
    \vspace{.5\baselineskip}

  \tikzexternalenable%
  \tikzsetnextfilename{mimo_msd_freq_g1_legend}%
  \filemodCmp{graphics/mimo_msd_freq_g1_legend.tikz}{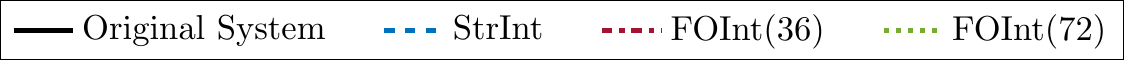}%
    {\tikzset{external/remake next}}{}%
  \begin{tikzpicture}[
  every axis/.append style = {
    scaled x ticks = false,
    x tick label style = {/pgf/number format/fixed},
    cycle list name = plotlist}
  ]
  
  \begin{axis}[%
    hide axis,
    scale only axis,
    width = 1mm,
    legend columns = 4, 
    legend style = {
    at     = {(0,0)},
    anchor = center,
    /tikz/every even column/.append style = {column sep = 0.5cm}}]
    \pgfplotsinvokeforeach{1,...,4}{\addplot coordinates {(0,0)};}
        
    \addlegendentry{Original System};
    \addlegendentry{StrInt};
    \addlegendentry{FOInt(36)};
    \addlegendentry{FOInt(72)};
  \end{axis}
\end{tikzpicture}%
  \tikzexternaldisable%

    \caption{Frequency domain results of the first transfer functions for the
      MIMO damped mass-spring system.}
    \label{fig:mimo_msd_freq_g1}
  \end{center}
\end{figure}

As in \Cref{sec:siso_msd}, we compare the structure-preserving 
interpolation method (StrInt) with the unstructured one, using the first-order 
realization of~\cref{eqn:msd_mimo} (FOInt).
For the construction of StrInt and FOInt(36), we choose
$\pm \texttt{logspace(-4, 4, 3)} \iu$ as interpolation points for the first 
transfer function and $\pm \texttt{logspace(-3, 3, 3)} \iu$ for the second one.
Additionally, we construct another first-order approximation, FOInt(72), twice 
as large as the structured interpolation by taking
$\pm \texttt{logspace(-4, 4, 6)} \iu$ and $\pm \texttt{logspace(-3, 3, 6)} \iu$, 
as interpolation points for the first and second transfer functions, 
respectively.
Also, we restrict ourselves again to a one-sided projection as in part (a) of 
\Cref{thm:mimovw} by setting $W = V$, which yields the reduced order
$r = 36$ for StrInt and FOInt(36), and $r = 72$ for FOInt(72).

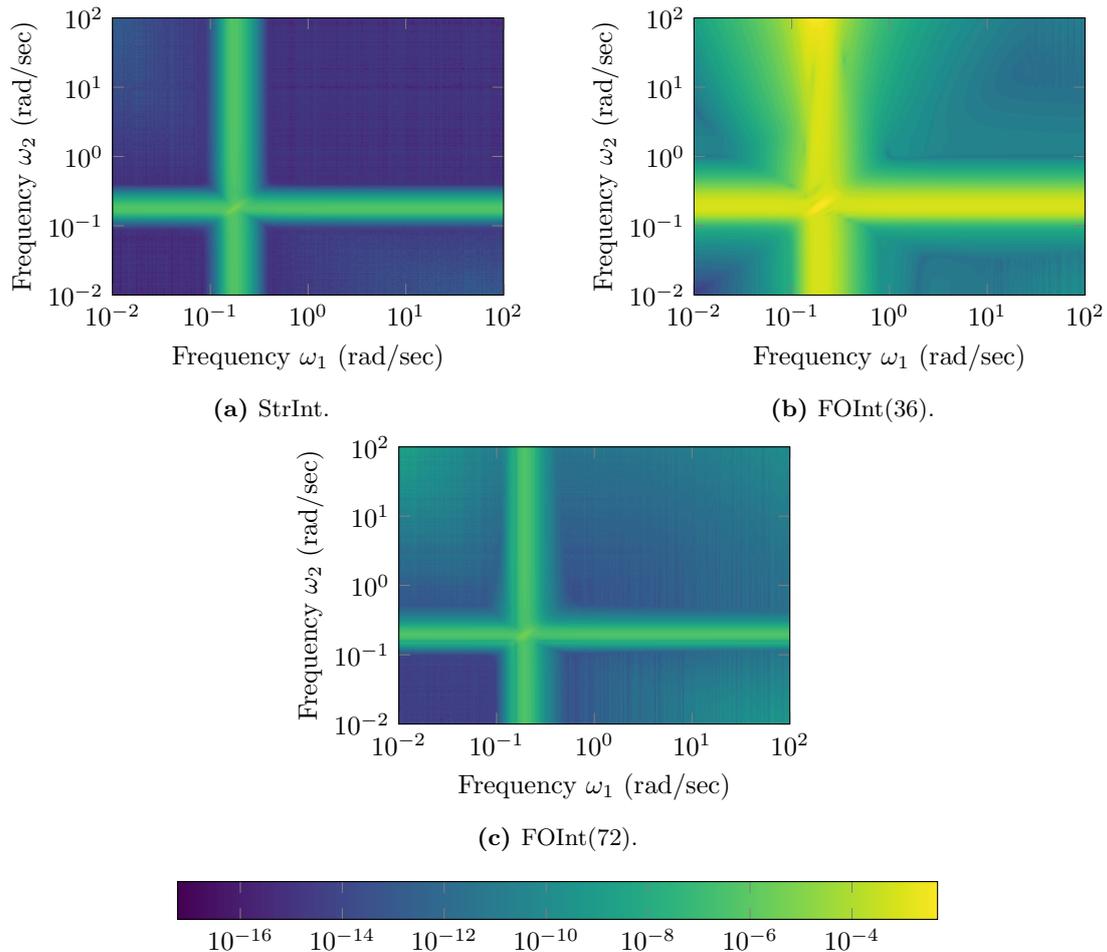
\begin{figure}[tb]
  \begin{center}
    \begin{subfigure}[t]{.49\textwidth}
      \begin{center}
  \tikzexternalenable%
  \tikzsetnextfilename{mimo_msd_freq_g2_err_so}%
  \filemodCmp{graphics/mimo_msd_freq_g2_err_so.tikz}{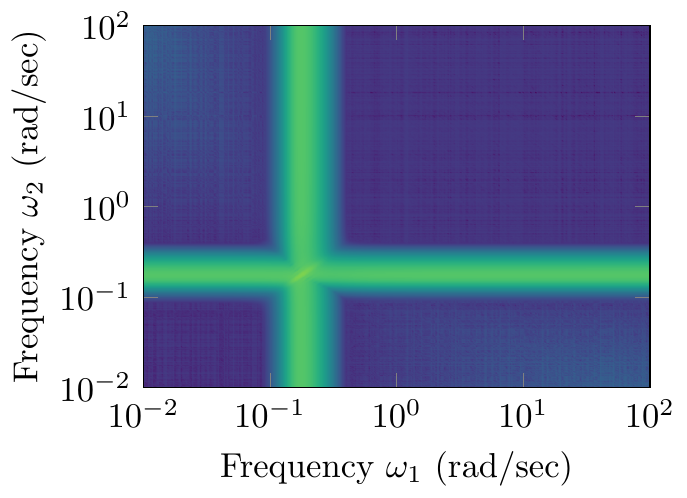}%
    {\tikzset{external/remake next}}{}%
  \begin{tikzpicture}
  \begin{loglogaxis}[
    view   = {0}{90},
    width  = .7\textwidth,
    height = .5\textwidth,
    scale only axis,
    axis on top,
    xmin   = 1e-2,
    xmax   = 1e+2,
    ymin   = 1e-2,
    ymax   = 1e+2,
    xtick  = {1e-2, 1e-1, 1e0, 1e+1, 1e+2},
    ytick  = {1e-2, 1e-1, 1e0, 1e+1, 1e+2},
    xminorticks = false,
    yminorticks = false,
    xlabel = {Frequency $\omega_{1}$ (rad/sec)},
    ylabel = {Frequency $\omega_{2}$ (rad/sec)},
    ylabel style = {yshift = -.3em},
    scaled x ticks = false,
    x tick label style = {/pgf/number format/fixed}]
        
      \addplot graphics[xmin = 1e-2, xmax = 1e+2, ymin = 1e-2, ymax = 1e+2]
        {graphics/data/mimo_msd_freq_g2_err_so.pdf};
            
  \end{loglogaxis}
\end{tikzpicture}%
  \tikzexternaldisable%

        \subcaption{StrInt.}
      \end{center}
    \end{subfigure}
    \hfill
    \begin{subfigure}[t]{.49\textwidth}
      \begin{center}
  \tikzexternalenable%
  \tikzsetnextfilename{mimo_msd_freq_g2_err_fo1}%
  \filemodCmp{graphics/mimo_msd_freq_g2_err_fo1.tikz}{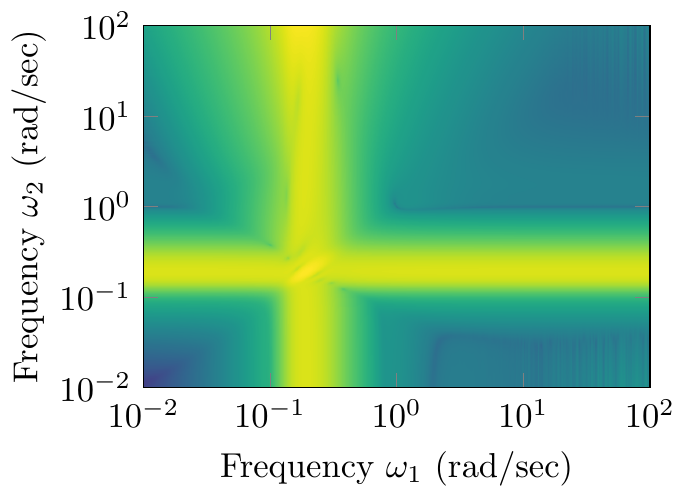}%
    {\tikzset{external/remake next}}{}%
  \begin{tikzpicture}
  \begin{loglogaxis}[
    view   = {0}{90},
    width  = .7\textwidth,
    height = .5\textwidth,
    scale only axis,
    axis on top,
    xmin   = 1e-2,
    xmax   = 1e+2,
    ymin   = 1e-2,
    ymax   = 1e+2,
    xtick  = {1e-2, 1e-1, 1e0, 1e+1, 1e+2},
    ytick  = {1e-2, 1e-1, 1e0, 1e+1, 1e+2},
    xminorticks = false,
    yminorticks = false,
    xlabel = {Frequency $\omega_{1}$ (rad/sec)},
    ylabel = {Frequency $\omega_{2}$ (rad/sec)},
    ylabel style = {yshift = -.3em},
    scaled x ticks = false,
    x tick label style = {/pgf/number format/fixed}]
        
      \addplot graphics[xmin = 1e-2, xmax = 1e+2, ymin = 1e-2, ymax = 1e+2]
        {graphics/data/mimo_msd_freq_g2_err_fo1.pdf};
            
  \end{loglogaxis}
\end{tikzpicture}%
  \tikzexternaldisable%

        \subcaption{FOInt(36).}
      \end{center}
    \end{subfigure}
    
    \begin{subfigure}[t]{.49\textwidth}
      \begin{center}
  \tikzexternalenable%
  \tikzsetnextfilename{mimo_msd_freq_g2_err_fo2}%
  \filemodCmp{graphics/mimo_msd_freq_g2_err_fo2.tikz}{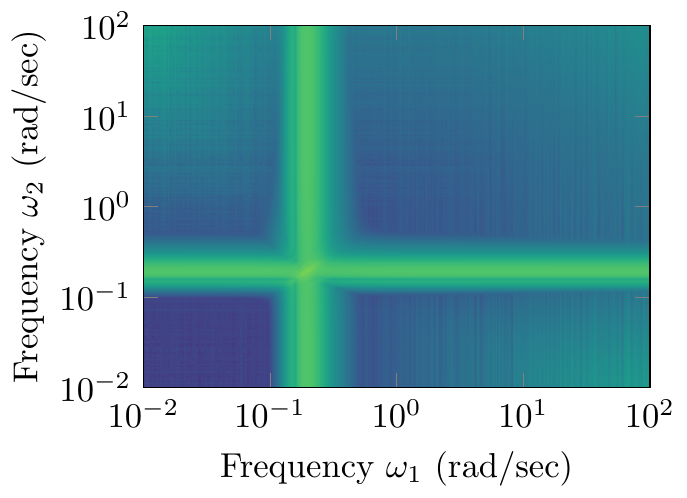}%
    {\tikzset{external/remake next}}{}%
  \begin{tikzpicture}
  \begin{loglogaxis}[
    view   = {0}{90},
    width  = .7\textwidth,
    height = .5\textwidth,
    scale only axis,
    axis on top,
    xmin   = 1e-2,
    xmax   = 1e+2,
    ymin   = 1e-2,
    ymax   = 1e+2,
    xtick  = {1e-2, 1e-1, 1e0, 1e+1, 1e+2},
    ytick  = {1e-2, 1e-1, 1e0, 1e+1, 1e+2},
    xminorticks = false,
    yminorticks = false,
    xlabel = {Frequency $\omega_{1}$ (rad/sec)},
    ylabel = {Frequency $\omega_{2}$ (rad/sec)},
    ylabel style = {yshift = -.3em},
    scaled x ticks = false,
    x tick label style = {/pgf/number format/fixed}]
        
      \addplot graphics[xmin = 1e-2, xmax = 1e+2, ymin = 1e-2, ymax = 1e+2]
        {graphics/data/mimo_msd_freq_g2_err_fo2.pdf};
            
  \end{loglogaxis}
\end{tikzpicture}%
  \tikzexternaldisable%

        \subcaption{FOInt(72).}
      \end{center}
    \end{subfigure}
    \vspace{.5\baselineskip}

  \tikzexternalenable%
  \tikzsetnextfilename{mimo_msd_freq_g2_legend}%
  \filemodCmp{graphics/mimo_msd_freq_g2_legend.tikz}{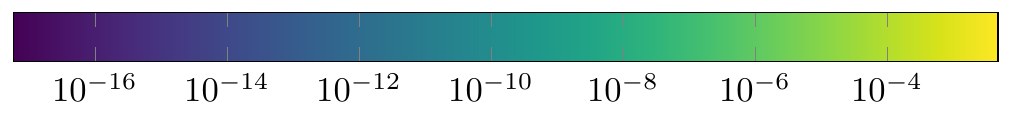}%
    {\tikzset{external/remake next}}{}%
  \begin{tikzpicture}
  \node[draw = none, minimum width = 0cm, inner sep = 0cm](start){};
  \node(leg) at (start.north east) [anchor = north west]{\tikz
  \begin{axis}[%
    hide axis,
    scale only axis,
    width  = 10cm,
    height = .1cm,
    point meta min = -17.2325,
    point meta max = -2.3211,
    colorbar,
    colorbar horizontal,
    colorbar style = {
      xticklabel = $10^{\pgfmathparse{\tick}
        \pgfmathprintnumber\pgfmathresult}$,
      at = {(.5, 0)},
      anchor = north},
    scaled x ticks = false,
    x tick label style = {/pgf/number format/fixed}]
  \end{axis};};
  \node[draw = none, minimum width = .0cm, inner sep = 0cm](end)
    at (leg.north east) [anchor = north west]{};
\end{tikzpicture}%
  \tikzexternaldisable%

    \caption{Relative errors of the second transfer functions for the
      MIMO damped mass-spring system.}
    \label{fig:mimo_msd_freq_g2}
  \end{center}
\end{figure}

\Cref{fig:mimo_msd_time} shows the results in time domain, where
we have chosen the input signal
\begin{align*}
  u(t) & = \begin{bmatrix} \sin(200 t) + 200 \\ -\cos(200 t) - 200 \end{bmatrix}
\end{align*}
and measured point-wise the relative errors as
\begin{align*}
  \frac{\lVert y(t) - \hy(t) \rVert_{2}}{\lVert y(t) \rVert_{2}},
\end{align*}
for $t \in [0, 100]$.
The different lines in \Cref{fig:mimo_msd_time_sim} with the same color
result from the two system outputs.
In contrast to the SISO case, the linear part of the larger unstructured 
approximation (FOInt(72)) is not asymptotically stable anymore, which leads to 
the fast diverging behavior in the time simulation.
The other first-order approximation (FOInt(36)) has a stable linear part but, as
in the SISO case, is not able to produce stable results in the time simulation.
StrInt again approximates the system's behavior accurately in the considered time
range and, by using one-sided projection, resembles the mechanical structures
of the original system.
\Cref{fig:mimo_msd_freq_g1,fig:mimo_msd_freq_g2} show the 
results of the approximations for the first two transfer functions, where
the relative errors are computed by
\begin{align*}
  \begin{aligned}
    \frac{\lVert G_{1}(\omega_{1} \iu) - \hG_{1}(\omega_{1} \iu) 
      \rVert_{2}}{\lVert G_{1}(\omega_{1} \iu) \rVert_{2}} &&& \text{and} &
      \frac{\lVert G_{2}(\omega_{1} \iu, \omega_{2} \iu) -
      \hG_{2}(\omega_{1} \iu, \omega_{2} \iu) \rVert_{2}}{\lVert 
      G_{2}(\omega_{1} \iu, \omega_{2} \iu) \rVert_{2}},
  \end{aligned}
\end{align*}
for $\omega_{1}, \omega_{2} \in [10^{-2}, 10^{2}]$.
For both transfer function levels, we observe that FOInt(36) is not as accurate
as StrInt and FOInt(72), which both nicely approximate the transfer functions
except for higher frequencies, where the unstructured approximation seems
to have the same numerical drift-off effect as in the SISO case.


\section{Conclusions}%
\label{sec:conclusions}

We extended the structure-preserving interpolation framework to bilinear control 
systems.
First, we developed the subspace conditions for structured interpolation for  
single-input single-output systems, both for simple and Hermite interpolation. 
These results were extended to structured multi-input multi-output 
bilinear systems as well in the setting of full matrix interpolation.
The effectiveness of the proposed approach was illustrated for two 
structured bilinear dynamical systems: a mass-spring-damper system and a model 
with internal delay.
The theory developed here can be applied to a much broader class of structures 
than these two examples.

In our examples, we made the rather simple choice of logarithmically
equidistant interpolation points on the first two transfer function levels; thus 
the crucial problem of choosing good/optimal interpolation points remains open. 
This question is not fully resolved even for structure-preserving interpolation 
of linear dynamical systems. 
Another issue to further investigate is the rapidly-enlarging reduced-order 
dimension in case of the matrix interpolation approach for multi-input 
multi-output systems.
While in the linear case, tangential interpolation can be used to control
the growth of the basis, there is no uniform treatment of tangential 
interpolation for bilinear systems yet.
This issue will be studied in a separate work.


\section*{Acknowledgments}%
\addcontentsline{toc}{section}{Acknowledgments}

Benner and Werner were supported by the German Research Foundation
(DFG) Research Training Group 2297
\textquotedblleft{}MathCoRe\textquotedblright{}, Magdeburg, and the German
Research Foundation (DFG) Priority Program 1897: \textquotedblleft{}Calm, 
Smooth and Smart -- Novel Approaches for Influencing Vibrations by Means of 
Deliberately Introduced Dissipation\textquotedblright{}.
Gugercin was supported in parts by National Science Foundation under Grant No. 
DMS-1720257 and DMS-1819110.
Part of this material is based upon work supported by the National Science 
Foundation under Grant No. DMS-1439786 and by the Simons Foundation Grant No. 
50736 while Gugercin and Benner were in residence at the Institute for 
Computational and Experimental Research in Mathematics in Providence, RI, 
during the \textquotedblleft{}Model and dimension reduction in uncertain and 
dynamic systems\textquotedblright{} program.

We would like to thank Jens Saak for constructive discussions about the
used notation and writing down of this paper, and Igor Pontes Duff Pereira and
Ion Victor Gosea for providing MATLAB codes used for the generation of the
bilinear time-delay example.


\addcontentsline{toc}{section}{References}
\bibliographystyle{plainurl}
\bibliography{bibtex/myref}

\end{document}